\documentclass[11pt]{article}
\usepackage{rotating}
\usepackage{framed}
\usepackage{color}
\usepackage{multirow}
\usepackage{graphicx}
\usepackage{subfigure}
\usepackage{epsfig}
\usepackage{cite}
\usepackage{amsmath, amssymb,mathrsfs}
\usepackage{fancybox}
\usepackage{listings}
\usepackage{verbatim}
\usepackage{url}
\usepackage{esint}
\usepackage{fancyhdr}
\usepackage{booktabs}
\usepackage{caption}
\usepackage{dsfont}
\usepackage{amsthm}

\usepackage{mathtools}

\input{epsf.sty}
\newtheorem{remark}{Remark}[section]
\newtheorem{example}{Example}[section] 
\newtheorem{proposition}{Proposition}[section] 
 
\newtheorem{definition}{Definition}[section]

\newcommand{\sgn}{\mathop{\rm sgn}}

\newcommand{\remove}[1]{}

\mathtoolsset{showonlyrefs}

\title{\bf 
	Modelling the random spreading of fake news through a two-dimensional time-inhomogeneous birth-death process\footnote{\bf Paper accepted for publication on {\em Mathematical Methods in the Applied Sciences}, $\copyright$ John Wiley \& Sons Ltd.}}
\author{Antonio Di Crescenzo$^{(1)}$, Paola Paraggio$^{(1)}$
	\\
	\normalsize (1)  Department of Mathematics,
	\normalsize Universit\`a degli Studi di Salerno \\
	Via Giovanni Paolo II, 132, Fisciano (SA), Italy \\ 
	\normalsize Email: \{adicrescenzo, pparaggio\}@unisa.it 
}
\date{\today}

\begin{document}
	\maketitle
	\section*{Abstract}
	We consider a two-dimensional time-inhomogeneous birth-death process to model the time-evolution of fake news in a population. The two components of the process represent, respectively, (i) the number of individuals (say spreaders) who know the rumor and intend to spread it, and (ii) the number of individuals (say inactives) who have forgotten the rumor previously received. We employ the probability generating function-based approach to obtain the moments and the covariance of the two-dimensional process. We also analyze a new adimensional index to study the correlation between the two components. Some special cases are considered in which both the expected numbers of spreaders and inactives are equal to suitable sigmoidal curves, which are often adopted in modelling growth phenomena. Finally, we provide an application based on real data related to the diffusion of fake news, in which the optimal choice of the sigmoidal curve that fit the datasets is based on the minimization of the mean square error and of the relative absolute error.\\
	\textbf{Keywords:} Two-dimensional birth-death process, probability generating function, adimensional correlation index, fake-news spreading, 
	mean square error, relative absolute error.		\\
	\textbf{MSC Classification:} 60J80, 
	60J85, 
	91B62,
	91B70.
	\section{Introduction}
	In recent years, the interest in issues related to the spread of fake news and to information security in general has increased significantly. Mainly, the researchers study the dynamics of online news diffusion aiming to warn users against all those situations that can harm the security of their own data. Along this line, governments of several states are allocating funds for cybersecurity. Internet allows everyone to spread news, both true and false; therefore, it is up to the user to recognize its truthfulness. Most of fake news is often ambiguous and insidious enough to mislead even careful readers. Moreover, it is important to remember that the dissemination of this type of information is based on psychological assumptions and leverages the fear and the emotion of public opinion. Often the authors of fake news cite real events, selecting only evidence to support their thesis and exploit advertising techniques to go viral. The consequences linked to the spread of fake news can be multiple, such as: focusing attention only on certain issues, diverting public opinion only on specific facts, denigrating entire ethnic groups or triggering repercussions on choices in various contexts (think about healthcare choices during the COVID-19 pandemic). Scientific research in this sector branches out into various areas ranging from the recognition of visual content modified by artificial intelligence and therefore false, to mathematical modelling. Regarding this last topic, various mathematical models have been proposed in the literature, represented both by simple curves capable of mimicking the temporal diffusion of fake news and by systems of differential equations useful to describe the temporal evolution of certain categories of the population (see, for example, the papers of San Martin \cite{SanMartin2020}, Di Crescenzo \textit{et al.}\ \cite{DiCrescenzoetal2023a} or Giorno and Spina \cite{GiornoSpina2016}). 
Furthermore, to let the models become more realistic, several investigations involve stochastic processes which are solutions of special stochastic differential equations defined by adding a random component in the deterministic models. The recent works of Mahmoud \cite{Mahmoud2020} and Esmaeeli and Sajadi \cite{EsmaeeliSajadi2023} can be viewed in this sense, since the authors adopt discrete-time stochastic processes for the temporal diffusion of fake news in a community. Many of these models are inspired by epidemiological ones. Indeed, in several stochastic systems the population is supposed to be divided into compartments according to the characteristics of the individuals that compose them (see, for example, Martins and Pinto \cite{MartinsPinto2020} for a model for the diffusion of fake news, and also Giorno and Nobile \cite{GiornoNobile2023} for an epidemiological model). 
	\par
	In this work, along the lines of the previous studies we introduce a new stochastic model for the diffusion of fake news in a community. More precisely, we consider a two-dimensional birth-death process $(X(t),Y(t))$, in which $X(t)$ represents the number of spreaders, i.e.\ individuals who, at time $t$, possess the information and intend to spread it, while $Y(t)$ represents the number of inactives, i.e.\ individuals who, at time $t$, have forgotten the information previously received. The dynamics of the process $(X(t), Y(t))$ is governed by two types of transitions. Specifically, if $(X(t),Y(t))=(n,k)$ then (i) the spreaders may increase with a time-dependent intensity,  $n\lambda(t)$, and (ii) the spreaders may decrease and become inactive with a possibly different time-dependent intensity, $n\mu(t)$, too.
	The aim of this paper is threefold: \\
	\noindent (a) to provide exact results about the probability distribution and the moments of the proposed model for the dynamics of the fake news spreading, \\
	\noindent (b) to study conditions such that the mean trend of the population categories, i.e. spreaders and inactives, correspond to typical growth curves adopted in populations dynamics, \\
	\noindent (c) to apply the proposed model to real datasets concerning the diffusion of fake news. 
\par	
		{In particular, regarding (a), we provide explicit expressions of the probability generating function, of the conditional moments and of other indexes of variation of both the components $X(t)$ and $Y(t)$.
		We also study the mixed moment, the covariance and the correlation coefficient in some special cases, namely when the spreader intensity $\lambda(t)$ and the inactive intensity $\mu(t)$ are constant and when they are proportional, i.e. $\lambda(t)=\rho\mu(t)$, $\rho>0$.}
		 Moreover, the correlation is studied using a new adimensional correlation index defined as the ratio between the mixed moment $\mathsf E\left[X(t)\cdot Y(t)\right]$ and the product of the means $\mathsf E\left[X(t)\right]$ and $\mathsf E\left[Y(t)\right]$.
		 \par 
		  With reference to (b), we consider some special instances of the two-dimensional process in which the intensities are proportional and are chosen so that the spreaders (or the inactives) have conditional mean of sigmoidal type. Specifically, we select 6 kinds of sigmoidal growth curves for the spreaders, and two  different types of sigmoidal growth curves starting from the origin, since they turn out to be more suitable to model growth phenomena subject to 
		  self-adaptive mechanisms, such as the diffusion of fake news.
		  \par
		  Finally, regarding (c), we use the proposed two-dimensional birth-death process with proportional intensities to fit real data concerning online fake news propagation. The analysis of the mean square error and of the relative absolute error 
		  allows us to assess the appropriateness of the considered sigmoidal curves leading to the suitable choice for the datasets under investigation.
	\par	
	We remark that two-dimensional birth-death processes have often been used in applied contexts since they prove to be effective in modeling certain phenomena evolving in a discrete state-space. For instance, in Di Crescenzo and Martinucci \cite{DiCrescenzoMartinucci2008}, the authors consider a two-dimensional birth-death process with constant rates, they analyze a property of quasi-symmetry and propose applications in contexts linked to double-ended queues. Moreover, G\'omez-Corral and L\'opez-Garc\'ia \cite{GomezLopez2012}, \cite{GomezLopez2015} introduce a two-dimensional continuous-time Markov chain in which each component represents a species of a population and the transition rates are proportional to the size of at least  one of the two components. They study, also, the number of births and deaths in an extinction cycle and the number of direct descendants under the hypothesis of killing and reproductive strategies. The contexts of application of two-dimensional processes are multiple and range from epidemiology (see, for example, Griffiths and Greenhalgh \cite{GriffithsGreenhalsh2011} and Billard \cite{Billard1981}), to agriculture (see, for instance, Abdullahi \textit{et al.}\ \cite{Adbullahietal2019}, in which the authors review three existing two-dimensional stochastic models showing their usefulness in applications related to seed dispersals), from service systems (as in Krishna-Kumar \textit{et al.}\ \cite{KrishnaKumar2023} and Liang Luh \cite{LiangLuh2015}), to politics (see, for example, Mobilia \cite{Mobilia2013} for a generalized three-party constrained voter model accounting for the competition between a persuasive majority and a committed minority) and also fake news diffusion (see, for instance, Kapsikar \textit{et al.}\  \cite{Kapsikaretal2021} in which the authors propose a two-dimensional branching process to describe fake news diffusion among populations composed by individuals capable to identify fake news). Generally, the results regarding such multi-dimensional processes are of numerical or approximation nature (as in Domenech-Benlloch \cite{DomenechBenllochetsl2008}, Servi \cite{Servi2002}, Brandwajn \cite{Brandwajn}, Fayolle \textit{et al.}\ \cite{Fayolleetal}, Remiche \cite{Remiche} and M\o ller and Berthelsen \cite{MollerBerthelsen}). 
	\par 
	This is the structure of the paper. An in-depth analysis of the stochastic model is presented in Section \ref{Section2} where the state-space of $(X(t), Y(t))$ is determined, highlighting the existence of absorbing states, together with the differential equations governing the dynamics of transition probabilities. 
	Sections 5 and Section 6 are devoted to the analysis of two specific cases. More precisely, in Section 5  the case of constant rates is studied, whereas in Section 6 the case of time-dependent and porportional rates is analyzed. In these two particular cases, it is possible to determine the expression of the probability generating function in closed form. 
	In Section \ref{Section3} the corresponding probability generating function is obtained (for specific cases in closed form). Moreover, Section \ref{Section4} is devoted to the moments, so that in Section \ref{Section3.1}, the moments (of the first and second order) of the first component $X(t)$ of the process are determined. Section \ref{Section5} contains similar results for the moments of the second component $Y(t)$ of the process. In Section \ref{Sec:MixedMom}, we then compute the mixed moments to analyze the correlation between $X(t)$  and  $Y(t)$. {In particular, the correlation is studied using also the new adimensional correlation index defined as the ratio between the mixed moment and the product of the conditional means.} In Section \ref{Sec:special-cases}, we study special cases in which the time-dependent rates  $\lambda(t)$ and $\mu(t)$ are  proportional, when the mean of $X(t)$ and the mean of $Y(t)$ are equal to some sigmoidal growth curves of interest in the literature, such as the Gompertz curve, both in the classical and generalized versions, the logistic curve and certain generalizations, and the modified Korf curve. The Korf and the Mitscherlich curves are also employed for the mean of $Y(t)$ since they meet the customary assumption $Y(0)=0$. Finally, in Section \ref{Section8}, we present an application based on real data, related to the diffusion of 5 different fake news in an online social network. We investigate the goodness of the proposed sigmoidal models through their mean square errors and  relative absolute errors.
	\section{The stochastic model}\label{Section2}
	Let $\left\{(X(t), Y(t));t\ge0\right\}$ be a two-dimensional birth-death process, where $X(t)$ represents the number of \textit{spreaders}, i.e. individuals who, at time $t$, possess the information and have the intention to spread it. Moreover, $Y(t)$ represents the number of \textit{inactives}, i.e.\ individuals who, at time $t$, have forgotten the information previously received and do not have the intention of sharing the rumor anymore. We suppose that $X(0)=j\in\mathbb N$, i.e.\ the number of informed individuals is initially greater than $0$, otherwise there would  be no diffusion of information. Moreover, we assume $Y(0)=0$ so that at the beginning no one has already forgotten the rumor.
	For $n,n',k,k'\in\mathbb N_0$, 
	we denote by
	$$
	q_{(n,k);(n',k')}(t):=\lim_{h\to 0}\frac{1}{h}\mathsf P\left[(X(t+h), Y(t+h))=(n',k')| (X(t),Y(t))=(n,k)\right] 
	$$
	the transition intensities of the process $(X(t),Y(t))$. Thus, we assume that the dynamics of the process are governed by the following transition rates:
	\begin{equation}\label{trans-rates}
		\begin{aligned}
			&q_{(n,k);(n+1,k)}(t)=n\lambda(t),\qquad n,k\in\mathbb N_0,\\
			&q_{(n,k);(n-1,k+1)}(t)=n\mu(t),\qquad n\in\mathbb N,\;k\in\mathbb N_0,
		\end{aligned}
	\end{equation}
	where $\lambda(t),\mu(t)$ are two positive and integrable functions on any interval $(0,\tau)$, $\tau>0$. 
	Both intensities in \eqref{trans-rates} are linear in $n$, since at each instant any spreader may transmit the information to a new individual or, on the contrary, he/she may become inactive. 
	The function $\lambda(t)$ represents the intensity that, at time $t$, an informer shares the rumor with a new individual. Whereas, the function $\mu(t)$ is the intensity that, at time $t$, an informer forgets the rumor and becomes inactive.
	These functions are both time-dependent since we assume that the characteristics of the spread and of the forgetfulness of the news may change over time  owing to external factors.
	Hence, $\lambda(t)$ represents the spreader intensity pro capite whereas $n\lambda(t)$ is the (cumulative) spreader intensity. Similarly, $\mu(t)$ represents the inactivity density pro capite and $n\mu(t)$ is the (cumulative) inactivity intensity.
	Note that the choice of intensities with a linear dependence on $n$ rather than quadratic or higher orders is owing to the necessity to define a non-explosive model, i.e.\ one with finite time-limits. Note, additionally, that the intensities \eqref{trans-rates} depend only on the current number of spreaders and not on the number of inactives. These assumptions lead to a significant simplification of calculations but are also dictated by modeling choices. The present model, in fact, aims to consider an infinite population. Specifically, if one considers a closed population of $N$ individuals divided into 3 categories (i.e., spreaders, inactives, and susceptibles, i.e.\ individuals who have never heard the news), the spreaders' rate of increase would naturally be represented by $n(N-n-k)\lambda(t)$, 
where $n$ is the number of current spreaders, and $k$ that of current inactives. Consequently, we implicitly assume that $N$ is  large, so that  the quantity $N-n-k$ is approximately equal to $N$. Moreover, it is worth to notice that, for fixed $t$, the random variable $Y(t)$ counts the number of departures from the category of spreaders up to time $t$. A departure of a spreader can be interpreted as a recovering from the fake news. For this reason, in our model the intensity with which inactives increase depends exclusively on $n$. 
	We remark that,  owing to the Markovian nature of the process,
	$$\label{pigreco}
	\pi(t):=\frac{\lambda(t)}{\lambda(t)+\mu(t)},\qquad t\ge 0,
	$$
	represents the probability that, if a new transition occurs at time $t$, it increases the number of the spreaders by one.
	\par
	According to Eqs.\ \eqref{trans-rates},  the couples of $\mathbb N_0^2$ are not all reachable by the process $(X(t),Y(t))$. Indeed, it is easy to note that the set of the reachable states from the initial state $(j,0)$ is represented by
	\begin{equation*}
		S_j:=\left\{(n,k)\in\mathbb N_0^2:n\ge 0, k\ge \max\{0,j-n\}\right\},\qquad j\in\mathbb N.
	\end{equation*}
	The state-space $S_j$ and the possible transitions are plotted in Figure \ref{fig:Figure1}.
	
	\begin{figure}[t]
		\begin{center}
			\begin{picture}(341,216) 
				\put(20,45){\vector(1,0){260}} 
				\put(80,5){\vector(0,1){240}} 
				\put(265,30){\makebox(30,25)[t]{$n$}} 
				\put(55,235){\makebox(30,15)[t]{$k$}} 
				\put(120,125){\line(1,0){60}} 
				\put(120,85){\line(1,0){100}} 
				\put(160,125){\line(1,-1){80}} 
				\put(120,125){\line(1,-1){80}} 
				\put(160,45){\line(-1,1){80}} 
				\put(200,45){\line(-1,1){120}} 
				\put(240,45){\line(-1,1){160}} 
				
				%
				%
				\put(42,30){\makebox(50,10)[t]{$(0,0)$}} 
				\put(135,30){\makebox(50,10)[t]{$(j,0)$}} 
				\put(35,120){\makebox(50,10)[t]{$(0,j)$}} 
				%
				\put(160,45){\circle{4}}
				\put(200,45){\circle{4}}
				\put(240,45){\circle{4}}
				\put(120,85){\circle{4}}
				\put(160,85){\circle{4}}
				\put(200,85){\circle{4}}
				\put(120,125){\circle{4}}
				\put(160,125){\circle{4}}
				\put(80,165){\circle*{4}}
				\put(120, 42){\line(0,1){5}}
				\put(78, 85){\line(1,0){5}}
				\put(80,205){\circle*{4}}
				\put(120,165){\circle{4}}
				
				%
				\put(80,125){\circle*{4}}
				%
				\put(160,45){\vector(-1,1){23}} 
				\put(120,85){\vector(-1,1){23}} 
				\put(200,45){\vector(-1,1){23}} 
				\put(240,45){\vector(-1,1){23}} 
				
				\put(160,85){\vector(-1,1){23}} 
				
				\put(160,125){\vector(-1,1){23}} 
				\put(120,125){\vector(-1,1){23}} 
				
				\put(200,85){\vector(-1,1){23}} 
				\put(130,85){\vector(1,0){13}} 
				\put(170,85){\vector(1,0){13}} 
				\put(200,85){\vector(1,0){23}} 
				
				\put(170,45){\vector(1,0){13}} 
				\put(210,45){\vector(1,0){13}} 
				\put(210,75){\makebox(50,10)[t]{$\dots$}} 
				\put(160,125){\vector(1,0){23}} 
				\put(130,125){\vector(1,0){13}} 
				\put(170,116){\makebox(50,10)[t]{$\dots$}} 
				
				\put(120,165){\vector(1,0){23}} 
				\put(130,156){\makebox(50,10)[t]{$\dots$}} 
				\put(120,165){\vector(-1,1){23}}
			\end{picture} 
		\end{center}
		\caption{\small Representation of the state-space $S_j$ for $j=2$. The solid dots indicate the absorbing states, whereas the vectors represent all the admissible transitions.}
		\label{fig:Figure1}
	\end{figure}
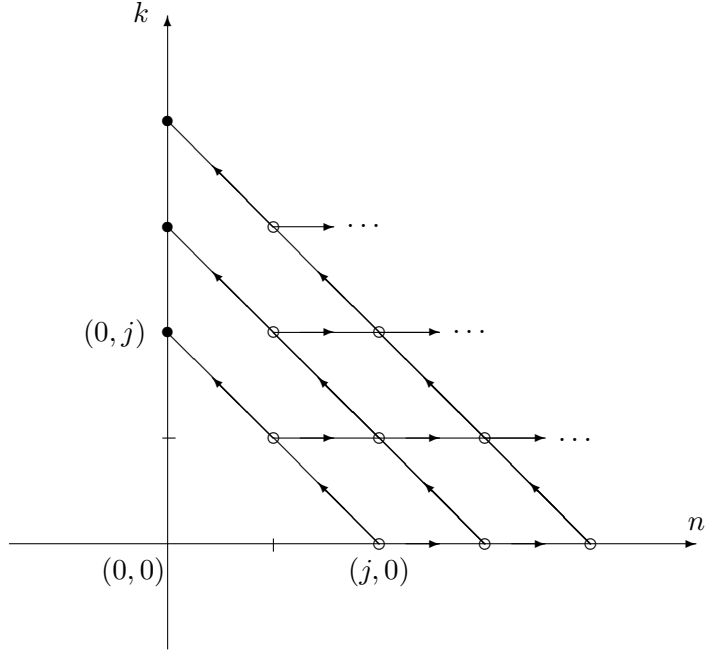
	
	%
	By setting
	\begin{equation*}
		p_{n,k}(t):=\mathsf P\left[X(t)=n, Y(t)=k\right],\qquad t\ge 0, 
	\end{equation*}
	owing to Eqs. \eqref{trans-rates}, the probabilities of the process $(X(t),Y(t))$ are governed by the following Kolmogorov equations, for $t>0$ and $j\in\mathbb N$,
	\begin{equation}\label{probeqdiff}
		\begin{aligned}
			&\frac{d}{dt}p_{n,k}(t)=p_{n-1,k}(t)(n-1)\lambda(t)+p_{n+1,k-1}(t)(n+1)\mu(t)-n\left(\lambda(t)+\mu(t)\right)p_{n,k}(t),\\
			&\qquad \qquad \qquad\qquad \qquad \qquad\qquad \qquad \quad \qquad \qquad \qquad  \text{for } n+k>j, \;k\ge 1,\; n\ge 1\\
			&\frac{d}{dt}p_{n,0}(t)=p_{n-1,0}(t)(n-1)\lambda(t)-n\left(\lambda(t)+\mu(t)\right)p_{n,0}(t),\qquad\quad \;\;\qquad  \text{for } n>j\\
			&\frac{d}{dt}p_{j,0}(t)=-j\left(\lambda(t)+\mu(t)\right)p_{j,0}(t),\\
			&\frac{d}{dt}p_{n,j-n}(t)=p_{n+1,j-n-1}(t)(n+1)\mu(t)-n\left(\lambda(t)+\mu(t)\right)p_{n,j-n}(t),\\
			&\qquad \qquad \qquad\qquad \qquad \qquad\qquad \qquad \quad \qquad \qquad \qquad\quad \text{for }1\le n<j,\; 1\le k<j\\
			&\frac{d}{dt}p_{0,k}(t)=p_{1,k-1}(t)\mu(t),\qquad \qquad \qquad \qquad\qquad \qquad \qquad \qquad \qquad\quad\text{for } k\ge j.
		\end{aligned}
	\end{equation}
	Since the initial state of the two-dimensional process is given by $(X(0),Y(0))=(j,0)$, the initial condition is 
	\begin{equation}\label{condini}
		p_{n,k}(0)=\delta_{n,j}\cdot\delta_{k,0},\qquad n\in\mathbb N,\; k\in\mathbb N_0,
	\end{equation}
	being $\delta_{ij}$ the Kronecker symbol. Note that  $X(t)$ is a linear time-inhomogeneous birth-death process, owing to Eq.\  \eqref{trans-rates}.
	\par 
	 For notational convenience, in the remainder of the paper we denote by $\mathsf{E}_j$ and $\mathsf{Var}_j$, respectively, the mean and the variance conditional on the initial state $(X(0),Y(0))=(j,0)$.

	\section{The probability generating function}\label{Section3}
	In order to obtain the probabilities $p_{n,k}(t)$ and the moments of the process $(X(t), Y(t))$, we consider the corresponding probability generating function (p.g.f.)
	\begin{equation}\label{prob-genfun}
		G_j(z_1,z_2,t):=\mathsf E_j \left[z_1^{X(t)} z_2^{Y(t)}\right]=\sum_{n=0}^{+\infty}\sum_{k=\max\{0,j-n\}}^{+\infty}p_{n,k}(t)z_1^nz_2^k, \qquad |z_1|\le 1,\;|z_2|\le1,\; t\ge 0.
	\end{equation}
	In the following proposition we provide the partial differential equation solved by the p.g.f. introduced in Eq.\ \eqref{prob-genfun}.
	\begin{proposition}
		The p.g.f. $G_j(z_1,z_2,t)$ of the two-dimensional birth-death process $(X(t),Y(t))$, with transition rates  \eqref{trans-rates}, solves the following partial differential equation
		\begin{equation}\label{PDE-G}
			\frac{\partial }{\partial t} G_j= \left[-z_1\lambda(t)\left(1-z_1\right)+\mu(t)(z_2-z_1)\right] \frac{\partial }{\partial z_1}G_j, \qquad |z_1|\le1,\; |z_2|\le1,\; t\ge 0,
		\end{equation}
		with initial condition given by $G_j(z_1,z_2,0)=z_1^j$.
	\end{proposition}
	\begin{proof}
		The equation given in Eq.\ \eqref{PDE-G} comes from the definition of $G_j(z_1,z_2,t)$ and Eqs.\ \eqref{probeqdiff} after straightforward calculations. The initial condition is an immediate consequence of Eq.\ \eqref{condini}.
	\end{proof}
The PDE given in Eq.\ \eqref{PDE-G} generally does not have a closed-form solution. However, as will be seen in Sections 5 and 6, it can be solved using classical methods when the rates are constant, or temporally non-homogeneous and proportional. Note that for the usual birth-death process, the probability generating function is such that 
\begin{equation}\label{relazione1}
	G_j(z_1,z_2,t)=\left(G_1(z_1,z_2,t)\right)^j,\qquad |z_1|<1,\; |z_2|<1,\; t\ge 0,\; j\in\mathbb N.
\end{equation}
 From Eq.\ \eqref{PDE-G}, it follows from Eq.\ \eqref{relazione1} then $G_1(z_1,z_2,t)$ satisfies the same PDE and the same initial condition with $j=1$. Therefore, despite the validity of the relation \eqref{relazione1}, the difficulty of the problem in question remains unchanged. It is also observed that the validity of the identity \eqref{relazione1} will be further confirmed in Sections 5 and 6, where the  probability generating function is available in closed form. 
	\begin{remark}
		Clearly, when $z_2\to 1$ in Eq.\ \eqref{prob-genfun} we return to p.g.f. of the the classical one-dimensional birth-death process with linear rates. Indeed, for $z_2\to 1$, Eq.\ \eqref{PDE-G} becomes
		\begin{equation}\label{PDE-G-X}
			\frac{\partial}{\partial t} G_j=(1-z_1)(\mu(t)-z_1\lambda(t)) \frac{\partial}{\partial z_1} G_j, \qquad |z_1|\le1,\;t\ge 0,
		\end{equation}
		which is the partial differential equation solved by the p.g.f. of a one-dimensional birth-death process with linear rates 
		(see Ricciardi \cite{Ricciardi} where the explicit solution of Eq.\ \eqref{PDE-G-X} is available). 
	\end{remark}
	\par
	
	\subsection{Extinction of the fake news}
	From the differential equations \eqref{probeqdiff} and the related assumptions, it is easy to notice the existence of absorbing states for the probabilities of the process $(X(t), Y(t))$.
	More in detail, the set of the absorbing states is represented by 
	\begin{equation}\label{Aj}
		A_j:=\left\{(0,k):k\ge j\right\}, \qquad X(0)=j\in\mathbb N.
	\end{equation}
	The absorption in $A_j$ corresponds to the eventuality that no individuals are left willing to spread the news, and therefore the diffusion of information ends.
	The absorption probability, namely the probability that the process $(X(t), Y(t))$ reaches the absorbing set $A_j$, defined as 
	$$p_{A_j}(t):=\mathsf{P}\left((X(t), Y(t))\in A_j\right).$$ 
	The probability $p_{A_j}(t)$ can be easily computed by means of the p.g.f. \eqref{prob-genfun}, since
		\begin{equation}\label{Pa}
			p_{A_j}(t)=\left.G_j(z_1,z_2,t)\right|_{z_1\to 0, z_2\to 1}, \qquad t\ge 0.
		\end{equation}
	\section{The conditional moments}\label{Section4}
	Hereafter, we derive some explicit results about the moments of the process $(X(t),Y(t))$, with special attention to the correlation.
	\par
	\subsection{The conditional moments of $X(t)$}\label{Section3.1}
	We remark that $X(t)$ is a  time non-homogeneous linear birth-death process with individual rates given by $\lambda(t)$ and $\mu(t)$, whose moments are already known in literature (see, for example, Eq.\ (7.48) and Eq.\ (7.49) of \cite{Ricciardi}).
		Specifically, the linear time-inhomogeneous birth-death process $X(t)$ has the conditional expected value and the conditional variance given respectively by
		\begin{equation}
			\label{mX}
			m_X(t;j):=\mathsf E_j[X(t)]=j\eta(t),\qquad t\ge 0,
		\end{equation}
	and 
	\begin{equation}\label{VarX}
		Var_X(t;j):=\mathsf{Var}_j[X(t)]=j\eta^2(t)\left[\frac{1}{\eta(t)}+2\phi_X(t)-1\right],\qquad t\ge 0,
	\end{equation}
	 with
		\begin{equation}\label{psi}
			\eta(t):=\exp\left\{\int_0^t (\lambda(\tau)-\mu(\tau))\textrm{d}\tau\right\},\qquad 
				\phi_X(t):=\int_0^t\frac{\lambda(\tau)}{\eta(\tau)}\textrm{d}\tau,\qquad t\ge 0.
		\end{equation}

	\subsection{The conditional moments of $Y(t)$}\label{Section5}
	In this section, we provide the expected value and the variance of the stochastic process $Y(t)$. We remark that $Y(t)$ represents the number of the inactives, i.e.\ the individuals who have forgotten the information previously received. 
	\begin{proposition}
		The stochastic process $Y(t)$ has the conditional expected value given by
		\begin{equation} \label{mY}
			m_Y(t;j):=\mathsf{E}_j\left[Y(t)\right]=j\left[\phi_Y(t)-\eta(t)+1\right],\qquad t\ge 0,
		\end{equation}
		where
		$$
		\phi_Y(t):=\int_0^t\lambda(\tau)\eta(\tau)\textrm{d}\tau,\qquad t\ge 0.
		$$
		\begin{proof}
			Differentiating both sides of Eq.\ \eqref{PDE-G} with respect to $z_2$, we obtain 
			\begin{equation}\label{derz2PGF}
				\frac{\partial^2}{\partial z_2\partial t}G_j=\left[-z_1\lambda(t)(1-z_1)+\mu(t)(z_2-z_1)\right] \frac{\partial ^2}{\partial z_2\partial z_1}G_j+\mu(t)\frac{\partial }{\partial z_1}G_j.
			\end{equation}
			Hence, by evaluating Eq.\ \eqref{derz2PGF} for $z_1,z_2\to 1$, we get the following ordinary differential equation
			\begin{equation*}
				\frac{d}{dt}m_Y(t;j)=m_X(t;j)\mu(t),\qquad m_Y(0;j)=0.
			\end{equation*}
			Therefore, the result comes after straightforward calculations.
		\end{proof}
	\end{proposition}
	In the following proposition we provide the second order moment of $Y(t)$ given the initial condition $(X(0),Y(0))=(j,0)$. 
	To this aim, we will need the conditional mixed moment given in Section \ref{Sec:MixedMom}.
	\begin{proposition}
		The conditional second order moment of the stochastic process $Y(t)$ is given by
		\begin{equation} \label{mom2Y}
			m_{2,Y}(t;j):=\mathsf{E}_j\left[Y(t)^2\right]=\int_0^t \mu(\tau)m_X(\tau;j)\left[1+2\gamma(\tau)\right]\textrm{d}\tau,\qquad t\ge 0,
		\end{equation}
		where 
		\begin{equation} \label{gamma}
			\gamma(t):=\int_0^t{\mu(\tau)\eta(\tau)} \left(2\phi_X(\tau)+j-1\right)\textrm{d}\tau, \qquad t\ge0.
		\end{equation}
		\begin{proof}
			By differentiating once again Eq.\ \eqref{derz2PGF}, we obtain
			\begin{equation} \label{der2z2PGF}
				\frac{\partial ^3}{\partial z_2^2\partial t}G_j\left[-z_1\lambda(t)(1-z_1)+\mu(t)(z_2-z_1)\right] \frac{\partial^3}{\partial z_2^2\partial z_1}G_j+2\mu(t)\frac{\partial ^2}{\partial z_2\partial z_1}G_j.
			\end{equation}
			Combining Eqs.\ \eqref{derz2PGF} and \eqref{der2z2PGF}, and evaluating the resulting equality for $z_1,z_2\to 1$, we obtain
			\begin{equation}
				\frac{d}{dt}m_{2,Y}(t;j)=\mu(t) m_X(t;j)+2\mu(t)m_{X\cdot Y}(t;j),\qquad m_{2,Y}(0;j)=0,
			\end{equation}
			with $m_{X\cdot Y}(t;j)$ given in Eq.\ \eqref{mXY} below. Hence,  Eq.\ \eqref{mom2Y} above easily follows.
		\end{proof}
	\end{proposition}
	\subsection{The conditional mixed moments}\label{Sec:MixedMom}
	Hereafter, we obtain the expression of the conditional expected value of $X(t)\cdot Y(t)$ and some other indexes to study the correlation of the processes $X(t)$ and $Y(t)$. 
	\begin{proposition}
		The conditional mixed moment of the first order of  $(X(t),Y(t))$  is 
		\begin{equation} \label{mXY}
			m_{X\cdot Y}(t;j)=\mathsf E_j[X(t)Y(t)]=m_X(t;j)\gamma(t),\qquad t\ge 0
		\end{equation}
		where $m_X(t;j)$ is given in Eq.\ \eqref{mX} and $\gamma(t)$ is defined in Eq.\ \eqref{gamma}.
		\begin{proof}
			By differentiating twice Eq.\ \eqref{PDE-G}, once with respect to $z_1$ and once with respect to $z_2$, we obtain the following equation
			\begin{equation} \label{derivataG}
				\begin{aligned}
					\frac{\partial^3}{\partial z_1\partial z_2\partial t} G_j&=\left[-z_1\lambda(t)(1-z_1)+\mu(t)(z_2-z_1)\right] \frac{\partial^3}{\partial z_2\partial z_1^2}G_j\\
					&+ \mu(t)\frac{\partial ^2}{\partial z_1^2}G_j+\left[-\lambda(t)-\mu(t)+2z_1\lambda(t)\right]\frac{\partial^2}{\partial z_2\partial z_1}G_j.
				\end{aligned}
			\end{equation}
			Therefore, by considering $z_1,z_2\to 1$, from Eq.\ \eqref{derivataG} we get the following ordinary differential equation
			\begin{equation*}
				\frac{d}{dt}m_{X\cdot Y}(t;j)-m_{X\cdot Y}(t;j)\left(\lambda(t)-\mu(t)\right)=\left(m_{2,X}(t;j)-m_X(t;j)\right)\mu(t),\qquad m_{X\cdot Y}(0;j)=0.
			\end{equation*}
			Hence, Eq.\ \eqref{mXY} comes after straightforward calculations and by considering Eqs.\ \eqref{mX} and \eqref{VarX}.
		\end{proof}
	\end{proposition}
	In order to get deeper information about the correlation of the random processes $X(t)$ and $Y(t)$, we can introduce and study the following index defined as the ratio between the mixed moment $m_{X\cdot Y}(t;j)$ and the product of the expected values $m_X(t;j)$ and $m_Y(t;j)$. 
	\begin{definition}\label{coeffr}
		The correlation index of $(X(t), Y(t))$ is the ratio between $m_{X\cdot Y}(t)$ and the product $m_{X}(t)\cdot m_Y(t)$, i.e.
		\begin{equation}\label{corr-index}
		r(t;j):=\frac{m_{X\cdot Y}(t;j)}{m_X(t;j)\cdot m_Y(t;j)}=\frac{\gamma(t)}{m_Y(t;j)},\qquad t\ge 0,
		\end{equation}
		where $\gamma(t)$ is defined in Eq.\ \eqref{gamma} and $m_Y(t;j)$ in Eq.\ \eqref{mY}. 
	\end{definition}
Clearly, the ratio $r(t;j)$ represents an adimensional correlation index, in the sense that 
$$
r(t;j)\begin{cases}
	<1,\qquad \text{if $X(t)$ and $Y(t)$ are negatively correlated}\vspace{-0.2cm}\\
	=1,\qquad \text{if $X(t)$ and $Y(t)$ are uncorrelated}\vspace{-0.2cm}\\
	>1,\qquad \text{if $X(t)$ and $Y(t)$ are positively correlated.}
\end{cases}
$$
 For the computation of the ratio $r(t;j)$ we do not need the mixed moment $m_{X\cdot Y}(t)$ but only the function $\gamma(t)$, given in \eqref{gamma}, and the conditional mean $m_Y(t;j)$. So the analysis of the correlation index $r(t;j)$ is recommended especially when the determination of the mixed moment $m_{X\cdot Y}(t)$  is complicated.
	\section{Constant intensities}\label{Section7}
	In this section, we analyze the process $(X(t), Y(t))$ in the time-homogeneous  case, i.e. when the individual intensities are constant, i.e.\ 
	$$\lambda(t)=\lambda>0,\qquad \mu(t)=\mu>0,\qquad t\ge 0.$$
	Let us start from the problem of finding the solution of Eq.\ \eqref{PDE-G}. 
	Note that Eq.\ \eqref{PDE-G} can be solved by using the method of characteristics (see Section 3.2 of Evans \cite{Evans}, for instance). Such method converts the partial differential equation  into an appropriate system of ordinary differential equations. More in detail, for the problem under analysis, from Eq.\ \eqref{PDE-G} we get the following system
	\begin{equation}\label{PDEsG-const}
		\begin{cases}
			\frac{\partial t}{\partial v}=1\\
			\frac{\partial z_1}{\partial v}=z_1(\lambda+\mu)-z_2\mu-z_1^2\lambda=:P(z_1)\\
			\frac{\partial z_2}{\partial v}=0\\
			\frac{\partial G}{\partial v}=0.
		\end{cases}
	\end{equation}
	From the first equation of \eqref{PDEsG-const} we have $t=v$. Moreover, since $(\lambda+\mu)^2-4z_2\lambda\mu>0$ for any $\lambda,\mu>0$, $|z_2|\le1$, then the polynomial $P(z_1)$ has two different real roots, namely
	$$
	\xi_{1,2}:=\xi_{1,2}(z_2)=\frac{\lambda+\mu\pm\sqrt{(\lambda+\mu)^2-4z_2\lambda\mu}}{2\lambda}, \qquad \xi_1<\xi_2.
	$$
	Consequently, the solution of the second equation of \eqref{PDEsG-const} is given by 
	\begin{equation}\label{z1-caso1}
		z_1(t)=\begin{cases}
			\displaystyle \frac{\xi_2(z_1-\xi_1)-\xi_1(z_1-\xi_2)e^{\lambda t (\xi_2-\xi_1)}}{z_1-\xi_1- (z_1-\xi_2)e^{\lambda t (\xi_2-\xi_1)}}, \qquad &\lambda\neq\mu\\
			\displaystyle \frac{z_1\left(e^{2\mu t \sqrt{1-z_2}}(1-\sqrt{1-z_2})-1-\sqrt{1-z_2}\right)+z_2\left(1-e^{2\mu t \sqrt{1-z_2}}\right)}{1-z_1-e^{2\mu t \sqrt{1-z_2}}(1-z_1+\sqrt{1-z_2})-\sqrt{1-z_2}},\qquad &\lambda=\mu.
		\end{cases}
	\end{equation}
 We remark that the above expression of $z_1(t)$, for $\lambda=\mu$, is obtained 
	 by an application of L'H\^{o}pital's rule to case for $\lambda\neq \mu$. This holds for various other equations given in the rest of the paper.  
	Finally, from the last equation of \eqref{PDEsG-const}, one has 
	\begin{equation} \label{pgf-const-expression}
	G_j(z_1,z_2,t)=\left[\widetilde G(z_1,z_2,t)\right]^j,\qquad t\ge 0,
	\end{equation}
	where 
	\begin{equation}\label{G-cost}
		\begin{aligned}
			\widetilde G(z_1,z_2,t)
			=\begin{cases}
				\displaystyle \frac{\xi_2(z_1-\xi_1)-\xi_1(z_1-\xi_2)e^{\lambda t (\xi_2-\xi_1)}}{z_1-\xi_1- (z_1-\xi_2)e^{\lambda t (\xi_2-\xi_1)}}, \quad &\lambda\neq\mu\\
				\displaystyle \frac{z_1\left(e^{2\mu t \sqrt{1-z_2}}(1-\sqrt{1-z_2})-1-\sqrt{1-z_2}\right)+z_2\left(1-e^{2\mu t \sqrt{1-z_2}}\right)}{1-z_1-e^{2\mu t \sqrt{1-z_2}}(1-z_1+\sqrt{1-z_2})-\sqrt{1-z_2}},\quad &\lambda=\mu,
			\end{cases}
		\end{aligned}
	\end{equation}
	for any $t\ge 0$. It is worth to notice that, owing to Eq.\ \eqref{prob-genfun}, $G_j(1,1,t)=1$ as expected. We remark that also in the time-homogeneous case, expanding the p.g.f. as a power series with respect to $z_1$ and $z_2$ is complicated, since the dependence on the variables $z_i$, $i=1,2$, is quite intricate. Therefore, providing explicit formulae for the probabilities $p_{n,k}(t)$ is not an easy task.
	\par 
	Although, as previously mentioned, finding the closed-form expression of the probabilities $p_{n,k}(t)$ is complicated, it may be interesting to analyze $p_{0,k}(t)$ for $t\to+\infty$. This quantity represents the portion of individuals who have had contact with the fake news in the long term by the end of the diffusion. It is analogous to the size of the outbreak which in the epidemiological context represents the portion of individuals who have been infected by the epidemic in the long term.
	From Eq.\ \eqref{prob-genfun}, we have 
	$$
	G_j(0,z_2,t)=\sum_{k=j}^{+\infty} p_{0,k}(t)z_2^k,\qquad t\ge 0,
	$$
	hence 
	$$
	\lim_{t\to+\infty} p_{0,k}(t)=\lim_{t\to+\infty} \frac{1}{k!} \left.\frac{\partial^k }{\partial z_2^k}G_j(0,z_2,t)\right|_{z_2\to 0},\qquad k\ge j.
	$$
	In this specific case, from Eq.\ \eqref{G-cost}, the following limit holds
	$$
	\lim_{t\to+\infty}G_j(0,z_2,t)= \lim_{t\to+\infty}\left(\widetilde G(0,z_2,t)\right)^j=\xi_1^j.
	$$
	By considering the expansion in power series of $\xi_1^j$, we finally obtain 
	\begin{equation}\label{p_0kinf}
	\lim_{t\to+\infty} p_{0,k}(t)=\left(\frac{\lambda\mu}{(\lambda+\mu)^2}\right)^k\left(\frac{\lambda+\mu}{\lambda}\right)^j\left(\binom{2k-j-1}{k-1}-\binom{2k-j-1}{k}\right),\qquad k\ge j.
	\end{equation}
	We remark that a different approach useful to come to similar results has been developed in Artalejo {\em et al}.\ \cite{Artalejo}	
	for the analysis of the number of proliferation events to reach the maximum clonal size in a stochastic model of cell proliferation. 
	Some plots of $\lim_{t\to+\infty}p_{0,k}(t)$ are given in Figure \ref{fig:Figura2maggio2024}.
	\begin{figure}[h]
		\centering
		 \includegraphics[scale=0.5]{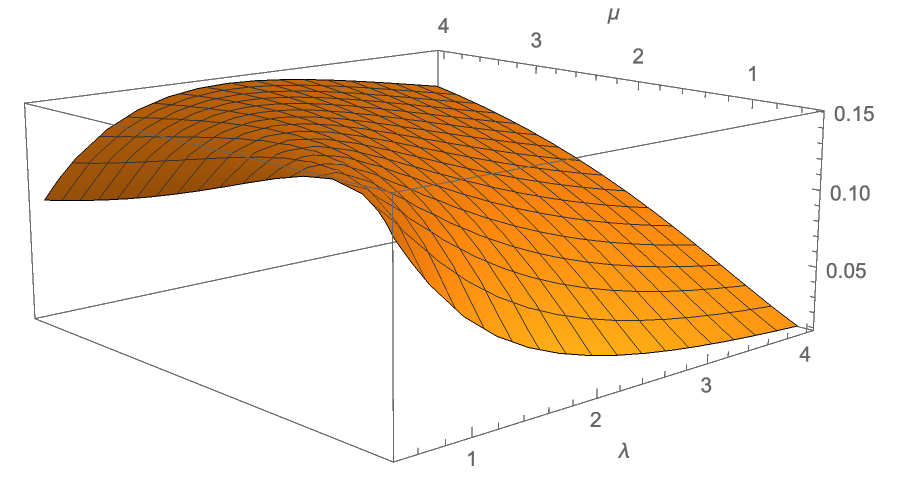}
		\caption{The probability $\lim_{t\to+\infty}p_{0,k}(t)$ for $j=1$ and $k=2$.}
		\label{fig:Figura2maggio2024}
	\end{figure}
	\par 
	Recalling Eq. \eqref{Pa}, in this case the absorption probability is given by
	\begin{equation}\label{probPAcost}
		p_{A_j}(t)=\begin{cases}
		\displaystyle	\left(\frac{\mu e^{(\lambda-\mu)t}-\mu}{\lambda e^{(\lambda-\mu)t}-\mu}\right)^j,\qquad &\lambda\neq \mu\\
		\displaystyle	\left(\frac{\mu t}{1+\mu t}\right)^j,\qquad &\lambda=\mu,
		\end{cases}
	\end{equation}
In Figure \ref{fig:Figure8} some plots of the absorption probability \eqref{probPAcost} are shown. 
Note that $p_{A_j}(t)$ converges to 1 when $\lambda\le \mu$, whereas when $\lambda>\mu$, indeed the following limit holds
\begin{equation}
\label{limite-PAj}
\lim_{t\to+\infty}p_{A_j}(t)=\begin{cases}
	\displaystyle \left(\frac{\mu}{\lambda}\right)^j,\qquad &\lambda>\mu\\
	1,\qquad &\lambda\le \mu.
\end{cases}
\end{equation}
\begin{figure}[h]
\centering
\subfigure[]{\includegraphics[scale=0.5]{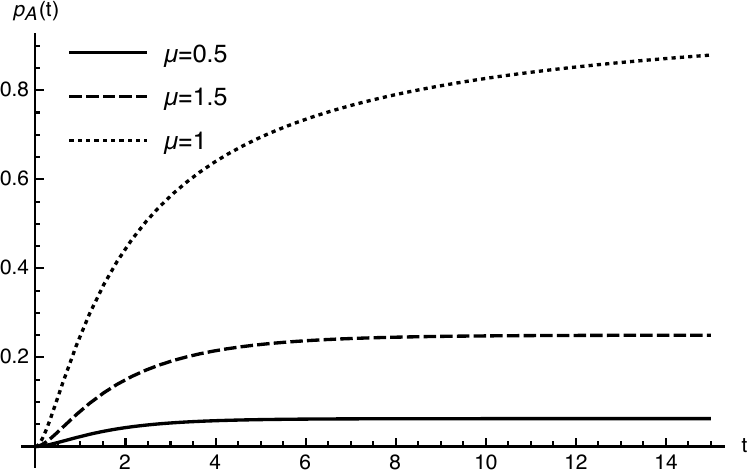}}\qquad
\subfigure[]{\includegraphics[scale=0.5]{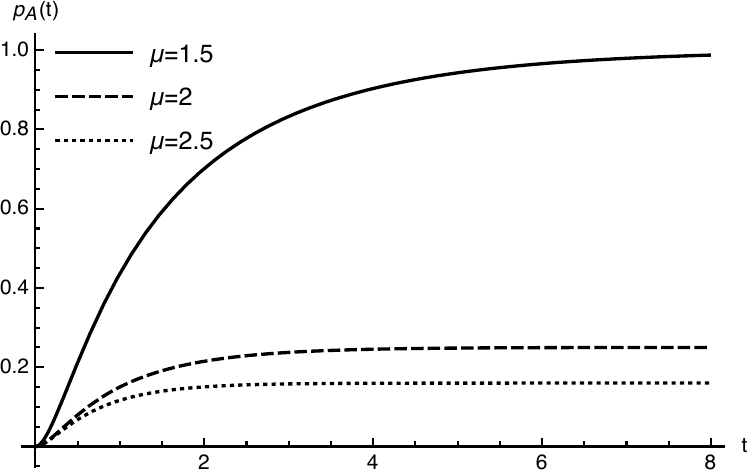}}
\caption{The absorption probability $p_{A_j}(t)$ for $\lambda=1$ and $j=1$.}
\label{fig:Figure8}
\end{figure}

Note that Eq.\ \eqref{limite-PAj} is in agreement with Eq.\ \eqref{p_0kinf}. Indeed, from Eq.\ \eqref{p_0kinf}, it follows that  
$
\lim_{t\to+\infty}\sum_{k=j}^{+\infty}p_{0,k}(t)=\lim_{t\to+\infty}p_{A_j}(t),
$ as expected. 
Moreover, we remark that the results given in Eqs.\ \eqref{probPAcost} and   \eqref{limite-PAj} are not fully novel, indeed they correspond to $p_0(t)$ and  $\lim_{t\to+\infty}p_0(t)$ given Section 6.4.3 of Allen \cite{Allen} for the time-homogeneous linear birth-death process.
\begin{example}
Disclosing an explicit expression for the probabilities $p_{0,k}(t)$, $k\in\mathbb N$, is complex, since obtaing the derivative of order $k$ of the p.g.f. $G_j(z_1,z_2,t)$ with respect to $z_2$ is a hard task. As example, we consider the most simple case in which the process $(X(t),Y(t))$ starts from $(1,0)$ and it is then absorbed at $(0,1)$. In particular, owing to Eq.\ \eqref{pgf-const-expression}, the probability $p_{0,1}(t)$ is given by:
$$
p_{0,1}(t)=\left.\frac{\partial}{\partial z_2}G_j(z_1,z_2,t)\right|_{z_1,z_2\to 0}=
	\displaystyle \frac{\mu (1-e^{-(\lambda+\mu)t})}{\lambda+\mu}
$$
for any $t\ge 0$.
See Figure \ref{fig:Figure9} for some plots of the probability $p_{0,1}(t)$. Note that $p_{0,1}(t)$ is increasing with respect to $t$ and $\mu$, whereas it is decreasing with respect to $\lambda$. 
Moreover, the following asymptotic behaviour holds
$$
\lim_{t\to+\infty} p_{0,1}(t)=
	\frac{\mu}{\lambda+\mu}.
$$
\begin{figure}[t]
	\centering
	\subfigure[]{\includegraphics[scale=0.5]{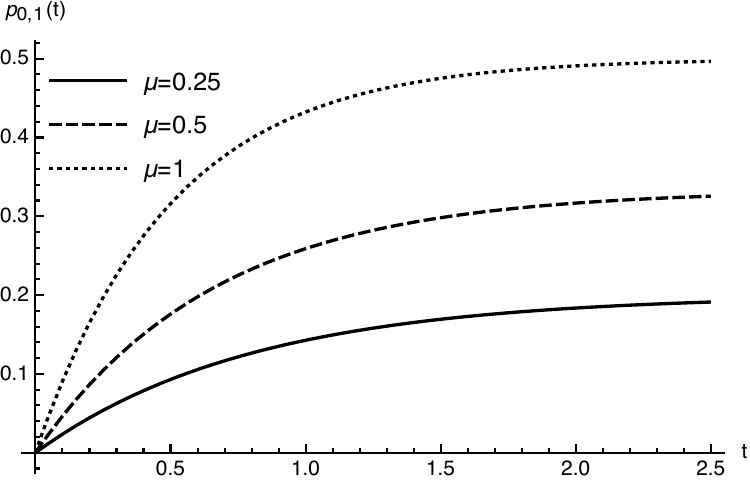}}\qquad
	\subfigure[]{\includegraphics[scale=0.5]{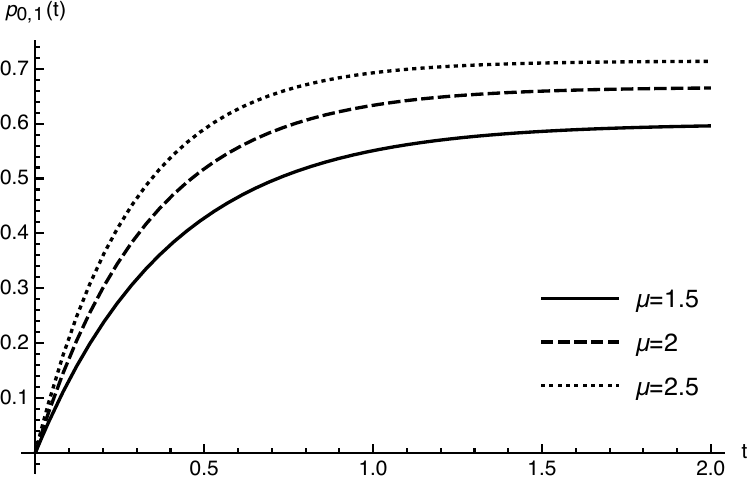}}
	\caption{The probability $p_{0,1}(t)$ for $\lambda=1$ and $j=1$.}
	\label{fig:Figure9}
\end{figure}
\end{example}
\subsection{The conditional moments of $X(t)$}
Regarding the linear time-homogeneous birth-death process $X(t)$, recalling Eqs. \eqref{mX} and \eqref{VarX}, the conditional expected value and the conditional variance of $X(t)$ are given by
		$$
		m_X(t;j)=je^{(\lambda-\mu)t}, \qquad Var_X(t;j)=
		\begin{cases}
			\displaystyle \frac{j(\mu+\lambda)}{\lambda-\mu}e^{(\lambda-\mu)t}\left(e^{(\lambda-\mu)t}-1\right),\qquad&\lambda\neq \mu\\
			2j\mu t,\qquad &\lambda=\mu,
		\end{cases}
		$$
		for any $t\ge 0$.
		Note that the following limits hold
		$$
		\lim_{t\to+\infty}m_X(t;j)=\begin{cases}
			+\infty \qquad &\lambda>\mu\\
			j,\qquad &\lambda=\mu\\
			0,\qquad &\lambda<\mu,
		\end{cases}\qquad
		\lim_{t\to+\infty}Var_X(t;j)=\begin{cases}
			+\infty \qquad &\lambda>\mu\\
			0,\qquad &\lambda\le\mu,
		\end{cases}
		$$	
		and
		$$
		\lim_{j\to+\infty}m_X(t;j)=\lim_{j\to+\infty}Var_X(t;j)=+\infty.
		$$
		Moreover, let us now determine the Fano factor $X(t)$ defined as the ratio between the variance and the mean. In this case, this index is given by
		\begin{equation}\label{DispersX}
			D_X(t;j):=\frac{Var_X(t;j)}{m_X(t;j)}=\begin{cases}
				\displaystyle \frac{(\mu+\lambda)(e^{(\lambda-\mu)t}-1)}{\lambda-\mu},\qquad &\lambda\neq\mu\\
				2\mu t,\qquad &\lambda=\mu,
			\end{cases}
		\end{equation}
		for any $t\ge0$. This is a dispersion measure that has the same units of the related data and it is often used for comparison with the Poisson process, for which the factor is equal to 1. We remark that $D_X(0,j)=0$, and that $D_X(t;j)$ is increasing with respect to $t$.   
		In particular, the following limits hold
		$$
		\lim_{t\to+\infty}D_X(t;j)=\begin{cases}
			\displaystyle \frac{\lambda+\mu}{\mu-\lambda},\qquad &\lambda<\mu\\
			+\infty, \qquad &\lambda\ge \mu.
		\end{cases}
		$$
		
		From Eq.\ \eqref{DispersX}, we have that the process $X(t)$ is underdispersed, i.e.\ $D_X(t;j)<1$, for all 
		$$
		t<\begin{cases}
			\log\frac{2\lambda}{\lambda-\mu}, \qquad &\lambda\neq \mu\\
			\frac{1}{2\mu}, \qquad &\lambda=\mu.
		\end{cases}
		$$
		 and it is overdispersed otherwise.
		Moreover, the coefficient of variation of $X(t)$ is given by
		$$
		\sigma_X(t;j):=	\frac{\sqrt{Var_X(t;j)}}{m_X(t;j)}=
		\begin{cases}
			\frac{e^{-\lambda t/2}}{j}
			\displaystyle \sqrt{\frac{(\lambda+\mu)j}{\lambda-\mu}\left(e^{\lambda t}-e^{\mu t}\right)},\qquad &\lambda\neq \mu\\
			\displaystyle \sqrt{\frac{2\mu t}{j}},\qquad &\lambda=\mu,
		\end{cases}
		$$
		for any $t\ge 0$. 
		We note that $\sigma_X(0;j)=0$ and that $\sigma_X(t;j)$ is increasing with respect to $t$. Moreover, the following limits hold
		$$
		\lim_{t\to +\infty} \sigma_X(t;j)=\begin{cases}
			0,\qquad &\lambda<\mu\\
			+\infty, \qquad &\lambda=\mu\\
			\displaystyle \sqrt{\frac{\lambda+\mu}{j(\lambda-\mu)}},\qquad &\lambda>\mu,
		\end{cases}\qquad \lim_{j\to +\infty}\sigma_X(t;j)=0.
		$$
		Therefore, for large times and for $\lambda <\mu$ or for large initial values, the process $X(t)$ turns out to be less variable around its mean. On the contrary, when $\lambda=\mu$, the process $X(t)$ is more variable around its mean.
Note that some results contained in this Section are well-known in the literature and can be found, for instance, in Section 6.4.3 of Allen \cite{Allen}.
			
		\subsection{The conditional moments of $Y(t)$}
	Now, we consider the process $Y(t)$. The conditional expected value of $Y(t)$ is, in time-homogeneous case, given by
		$$
		m_Y(t;j)=
		\begin{cases}
			\displaystyle	\frac{j\mu}{\lambda-\mu}\left(e^{(\lambda-\mu)t}-1\right),\qquad &\lambda\neq\mu\\
			j\mu t,\qquad &\lambda=\mu,
		\end{cases}
		$$
		for any $t\ge 0$.
		In Figure \ref{fig:Figure4}(a) some plots of the conditional expected value $m_Y(t;j)$ are provided. Note that in any case, $m_Y(t;j)$ is an increasing function but when $\lambda>\mu$ the concavity is upward, when $\lambda<\mu$ the concavity is downward. 
		Note that when $\lambda>\mu$, the process $X(t)$ shows an exponential growth trend and thus the growth trend of $m_Y(t;j)$ is increasing with upward concavity. Moreover, the following limits hold
		$$
		\lim_{t\to+\infty}m_Y(t;j)=\begin{cases}
			+\infty\qquad &\lambda\ge\mu\\
			\displaystyle\frac{j\mu}{\mu-\lambda},\qquad &\lambda<\mu,
		\end{cases}\qquad
		\lim_{j\to+\infty}m_Y(t;j)=+\infty.
		$$
		\begin{figure}[t]
			\centering
			\subfigure[]{\includegraphics[scale=0.5]{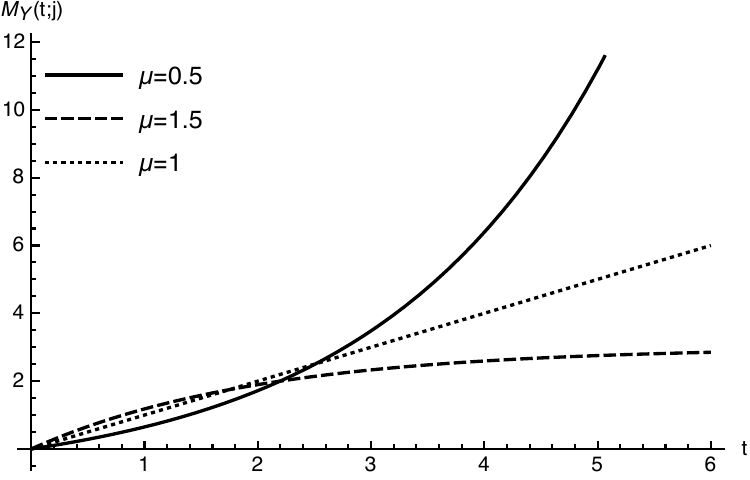}}\qquad
			\subfigure[]{\includegraphics[scale=0.5]{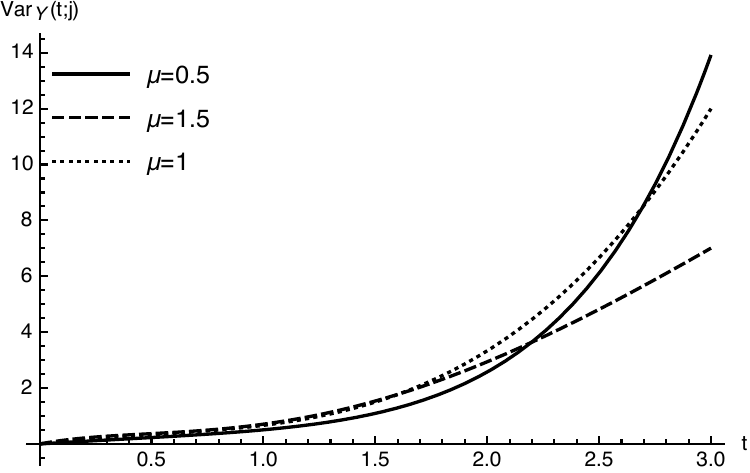}}
			\caption{The conditional expected value $m_Y(t;j)$ and the conditional variance $Var_Y(t)$ for $j=1$ and $\lambda=1$.}
			\label{fig:Figure4}
		\end{figure}
		Table \ref{tab:Tabella1} shows the sign of the difference between $m_X(t;j)$ and $m_Y(t;j)$. Note that when the spreader intensity $\lambda$ is at least equal to the double of the inactivity intensity $\mu$, then $m_X(t;j)$ is greater than $m_Y(t;j)$ for any $t\ge 0$, meaning that the expected number of spreaders is always greater than the expected number of forgetful individuals. Whereas, after an initial time interval $(0,\widetilde t)$ then $m_Y(t;j)>m_X(t;j)$ for any $t>\widetilde t>0$ when  $\lambda\ge \mu$, otherwise $m_Y(t;j)<m_X(t;j)$ when $\lambda<\mu$ with 
		\begin{equation}\label{tildet}
			\widetilde t=\begin{cases}
				\displaystyle\frac1\mu, \qquad &\lambda=\mu\\
				\displaystyle	\frac{1}{\lambda-\mu}\log\left(\frac{\mu}{2\mu-\lambda}\right),\qquad &\lambda\neq \mu.
			\end{cases}
		\end{equation}
		\begin{table}[t]
			\caption{The sign of the difference between the conditional means for different range of values of $\lambda$ and $\mu$ and $\widetilde t$ defined in Eq.\ \eqref{tildet}.}
			\label{tab:Tabella1}
			\centering
			\begin{tabular}{l|l} 
				&                                                                                                                                                                                                                                                                           \\ 
				\hline
				$0<\lambda<\mu$                   & $\displaystyle m_X(t;j)>m_Y(t;j) \iff t>\widetilde t$  \\
				$\mu\le\lambda<2\mu$ 			  & $\displaystyle m_X(t;j)>m_Y(t;j)\iff t<\widetilde t$  \\
				$\lambda\ge2\mu$ 				& $\displaystyle m_X(t;j)>m_Y(t;j)\qquad \forall t\ge 0$.                                                                                                                                                                                
			\end{tabular}
		\end{table}
	
		Furthermore,  the conditional second order moment of $Y(t)$ is given by
		$$
		\begin{aligned}
			m_{2,Y}(t;j)
			=\begin{cases}
				\displaystyle\frac{j\mu}{(\lambda-\mu)^3}\left(-\lambda^2-\lambda\mu+j\lambda\mu-j\mu^2+e^{2(\lambda-\mu)t}\mu(\lambda+j\lambda+\mu-j\mu)\right.\\
				\displaystyle\left.+e^{(\lambda-\mu)t}(\mu-\lambda)\left((2j-1)\mu+\lambda(4\mu t -1)\right)\right),\qquad &\lambda\neq \mu\\
				\displaystyle\frac{1}{3}j\mu t \left(3+3(j-1)\mu t + 2\mu^2t^2\right),\qquad &\lambda=\mu,
			\end{cases}
		\end{aligned}
		$$
		for any $t\ge 0$.
		In this case, the expression of the conditional variance of $Y(t)$  is also available in closed form and is given by
		$$
		\begin{aligned}
			&Var_Y(t;j):=\mathsf{Var}_j\left(Y(t)\right)\\
			&=\begin{cases}
				\displaystyle	\frac{j\mu}{(\lambda-\mu)^3}\left(-\lambda(\lambda+\mu)+e^{2(\lambda-\mu)t}\mu(\lambda+\mu)+e^{(\lambda-\mu)t}(\mu-\lambda)(-\mu+\lambda(4\mu t-1))\right),\quad&\lambda\neq \mu\\
				\displaystyle \frac{j\mu t}{3}\left(3-3\mu t +2\mu^2t^2\right),\quad &\lambda=\mu,
			\end{cases}
		\end{aligned}
		$$
		for any $t\ge 0$.
		Moreover, note that the following limits hold
		$$
		\lim_{t\to +\infty}Var_Y(t;j)=\begin{cases}
			\displaystyle\frac{j\lambda\mu (\lambda+\mu)}{(\mu-\lambda)^3},\qquad &\lambda<\mu\\
			\displaystyle +\infty,\qquad &\lambda\ge \mu,
		\end{cases}\qquad
		\lim_{j\to +\infty}Var_Y(t;j)=+\infty.
		$$
		Hence, the conditional mean $m_Y(t;j)$ is a significant index for the description of the evolution of $Y(t)$ when $\lambda<\mu$, since in this case the corresponding conditional variance is bounded.
		Some plots of the conditional variance $Var_Y(t;j)$ are shown in Figure \ref{fig:Figure4}(b).
		In this case, the Fano factor $D_Y(t;j)$ is given by
		$$
		D_Y(t;j)=
		\begin{cases}
			\displaystyle \frac{e^{2(\lambda-\mu)t}\mu(\lambda+\mu)-e^{(\lambda-\mu)t}(\lambda-\mu)(\lambda(4\mu t-1)-\mu)-\lambda(\lambda+\mu)}{(\lambda-\mu)^2(e^{(\lambda-\mu)t}-1)},\qquad &\lambda\neq \mu\\
			\displaystyle \frac{3-3\mu t+2\mu^2t^2}{3},\qquad &\lambda=\mu,
		\end{cases}
		$$
		for any $t\ge 0$. Moreover the following limit holds
		$$
		\lim_{t\to+\infty}D_Y(t;j)=\begin{cases}
			+\infty\qquad &\lambda\ge \mu\\
			\displaystyle \frac{\lambda(\lambda+\mu)}{(\mu-\lambda)^2},\qquad &\lambda<\mu.
		\end{cases}
		$$
		Furthermore, the coefficient of variation is given by
		$$
		\sigma_Y(t;j)=
		\begin{cases}\sqrt{\frac{
					\displaystyle e^{2(\lambda-\mu)t}\mu(\lambda+\mu)+e^{(\lambda-\mu)t}(\mu-\lambda)(\lambda(4\mu t -1)\lambda-\mu)-\lambda(\lambda+\mu)}{(\lambda-\mu)j\mu(e^{(\lambda-\mu)t}-1)^2}},\quad &\lambda\neq \mu\\
			\displaystyle \sqrt{\frac{3-3\mu t+2\mu^2t^2}{3j\mu t}},\quad &\lambda=\mu,
		\end{cases}
		$$
		for any $t\ge 0$. We point out that $\sigma_Y(t;j)$ is convergent for $\lambda\neq\mu$ as $t\to+\infty$, indeed we have
		$$
		\lim_{t\to+\infty}\sigma_Y(t;j)=\begin{cases}
			\sqrt{\frac{\lambda(\lambda+\mu)}{j\mu(\mu-\lambda)}},\qquad &\lambda<\mu\\
			+\infty,\qquad &\lambda=\mu\\
			\sqrt{\frac{\lambda+\mu}{j(\lambda-\mu)}},\qquad &\lambda>\mu.
		\end{cases}
		$$
		\subsection{Correlation between $X(t)$ and $Y(t)$}
In this section, we analyze the correlation between the two components of the process $(X(t), Y(t))$. {Considering that $X(t)$ represents the number of spreaders and $Y(t)$ the number of inactives, we expect that at the beginning of the diffusion of the fake news the two processes are negatively correlated, since it is reasonable that in the initial diffusion phase all individuals behave as gullible and therefore they have the intention of spreading the rumor. This intention is then gradually lost by the individuals and therefore we  expect that the two
	processes ultimately are positively correlated. }To this aim, let us start from the conditional expected value of $X(t)\cdot Y(t)$ which is given by $m_{X\cdot Y}(t;j)=j\mu \mathcal{M}(t;j)$, where
			$$
			\mathcal{M}(t;j)=
			\begin{cases}
				\displaystyle\frac{e^{(\lambda-\mu)t}}{\mu-\lambda}\left(2\lambda t-\frac{(e^{(\lambda-\mu)t}-1)(\lambda+j\lambda+\mu-j\mu)}{\lambda-\mu}\right),\qquad &\lambda\neq \mu\\
				\displaystyle	t (j+\mu t-1),\qquad&\lambda=\mu,
			\end{cases}
			$$
			for $t\ge 0$.
			In this case, it is possible to derive an explicit expression for the conditional covariance, which is given by 
			$$Cov(t;j):=\mathsf{Cov}_j(X(t),Y(t))=\mathsf{Cov}\left[(X(t), Y(t))|(X(0),Y(0))=(j,0)\right]=j\mu \mathcal C(t),\qquad t \ge 0$$ with
			$$
			\begin{aligned}
				\mathcal C(t):=\displaystyle\begin{cases}
					\displaystyle e^{(\lambda-\mu)t}\left( \frac{e^{(\lambda-\mu)t}(\lambda+\mu)- (\lambda+\mu+2\lambda(\lambda-\mu))t}{(\lambda-\mu)^2}\right),\qquad &\lambda\neq\mu\\
					\mu t-1,\qquad &\lambda=\mu,
				\end{cases}
			\end{aligned}
			$$
			and the conditional correlation is given by
			$$
			\begin{aligned}
				Corr(t;j)&:=\mathsf{Corr}_j((X(t),Y(t)))=\mathsf{Corr}\left[(X(t), Y(t))|(X(0),Y(0))=(j,0)\right]\\
				&=\begin{cases}
					\displaystyle \left[(\lambda+\mu)\left(e^{(\lambda-\mu)t}-1\right)-2(\lambda-\mu)t\right]\sqrt{\frac{\mu e^{(\lambda-\mu)t}}{(\lambda+\mu)(\alpha_1(t)+\alpha_2(t))}},\quad &\lambda\neq\mu\\
					\displaystyle \frac{\sqrt{3} (\mu t -1)}{\sqrt{6-6\mu t +4\mu^2t^2}},\qquad &\lambda=\mu,
				\end{cases}
			\end{aligned}
			$$
			with
			$$
			\alpha_1(t):=\lambda(\lambda+\mu)+e^{3(\lambda-\mu)t}\mu(\lambda+\mu)+e^{2(\lambda-\mu)t}(\lambda^2-\lambda\mu-2\mu^2+4\lambda\mu(\mu-\lambda)t),\qquad t\ge 0
			$$
			and 
			$$
			\alpha_2(t):=e^{(\lambda-\mu)t}\left(\mu^2+\lambda^2(4\mu t-2)-\lambda\mu(1+4\mu t)\right),\qquad t\ge 0.
			$$
			Note that in this case $Corr(t;j)$ does not depend on $j$.
			Moreover, the following limits hold
			$$
			\begin{aligned}
				&\lim_{t\to+\infty}Cov(t;j)=\begin{cases}
					0,\qquad &\lambda<\mu\\
					+\infty,\qquad &\lambda\ge \mu,
				\end{cases}\qquad
				&\lim_{t\to+\infty} Corr(t;j)=\begin{cases}
					0,\qquad &\lambda<\mu\\
					\frac{\sqrt{3}}{3}\qquad &\lambda=\mu\\
					1,\qquad &\lambda>\mu,
				\end{cases}\\
				&\lim_{t\to 0}Cov(t;j)=0,\qquad &\lim_{t\to 0}Corr(t;j)=-\sqrt{\frac{\mu}{\lambda+\mu}}.
			\end{aligned}
			$$
			Note that, for $t\to 0$, $Cov(t;j)\to 0$ whereas $Corr(t;j)\to -\sqrt{\frac{\mu}{\lambda+\mu}}$ since $Var_X(t;j)\to 0$ and $Var_Y(t;j)\to 0$. Hence, at the beginning the processes $X(t)$ and $Y(t)$ are negatively correlated.
			In Figure \ref{fig:Figure6}, some plots of the conditional covariance $Cov(t;j)$ and of the conditional correlation coefficient $Corr(t;j)$ are provided. {Note that, in any case, the processes $X(t)$ and $Y(t)$ are negatively correlated for small $t$, and are positively correlated or uncorrelated for large $t$. More in detail, when $t\to+\infty$ the processes are positively correlated if $\lambda\ge \mu$, and uncorrelated otherwise.
				These results can also be justified from a modeling point of view. At the beginning of the spread of fake news, it is reasonable to believe that all individuals in the population are gullible and are therefore inclined to spread the news. On the contrary, when the rumor has spread sufficiently, individuals tend to be skeptical and lose interest in spreading the fake news. 
			Regarding the covariance, we note that at the beginning it tends to 0 and when $t\to+\infty$ and $\lambda<\mu$, the covariance tends to 0, otherwise it tends to $+\infty$. Furthermore, we notice that
			when $\mu\le 1/2$ or when $\mu>1/2, \lambda<\mu(2\mu+1)/(2\mu-1)$ it results that $\left.\frac{\mathrm d}{\mathrm dt}Cov(t;j)\right|_{t\to 0}<0$, otherwise $\left.\frac{\mathrm d}{\mathrm dt}Cov(t;j)\right|_{t\to 0}>0$. Hence, when $\mu\le 1/2$ or when $\mu>1/2, \lambda<\mu(2\mu+1)/(2\mu-1)$, the covariance between the processes is negative in a time interval close to 0.\\
		}
			\begin{figure}[t]
				\centering
				\subfigure[]{\includegraphics[scale=0.5]{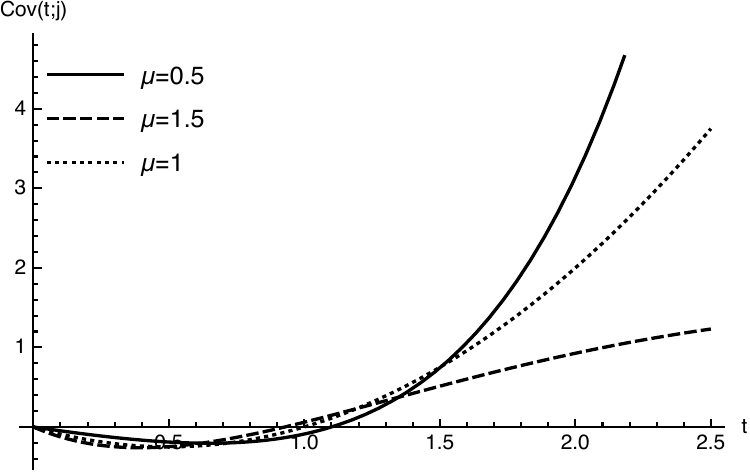}}\qquad
				\subfigure[]{\includegraphics[scale=0.5]{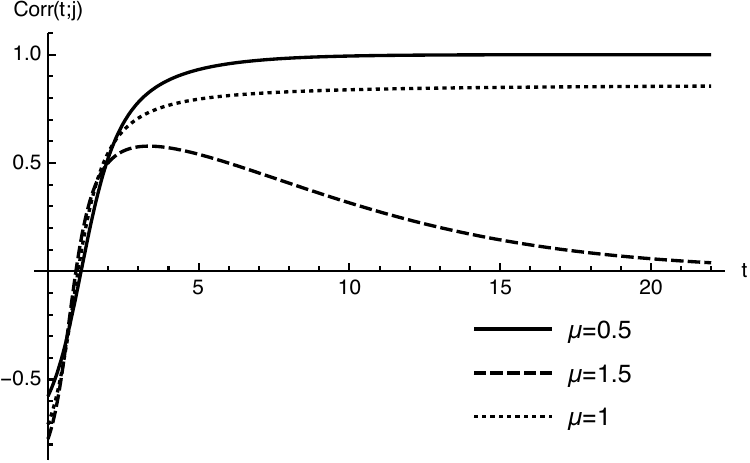}}
				\caption{The conditional covariance  $Cov(t;j)$ and the conditional correlation coefficient $Corr(t;j)$ for $j=1$ and $\lambda=1$.}
				\label{fig:Figure6}
			\end{figure}
			The coefficient $r(t;j)$, introduced in Definition \ref{coeffr} has the following expression
			\begin{equation}\label{rCostanti}
			r(t;j)=\begin{cases}
				\displaystyle1+\frac{1}{j}\left[\frac{\lambda+\mu}{\lambda-\mu}-\frac{2\lambda t }{e^{(\lambda-\mu)t}-1}\right],\qquad &\lambda\neq \mu\\
				\displaystyle 1+\frac{\mu t -1}{j},\qquad &\lambda=\mu,
			\end{cases}
			\end{equation}
			for any $t\ge 0$.
			Some plots of the ratio $r(t;j)$ are given in Figure \ref{fig:Figure7}.
			\begin{figure}[t]
				\centering
				\subfigure[]{\includegraphics[scale=0.5]{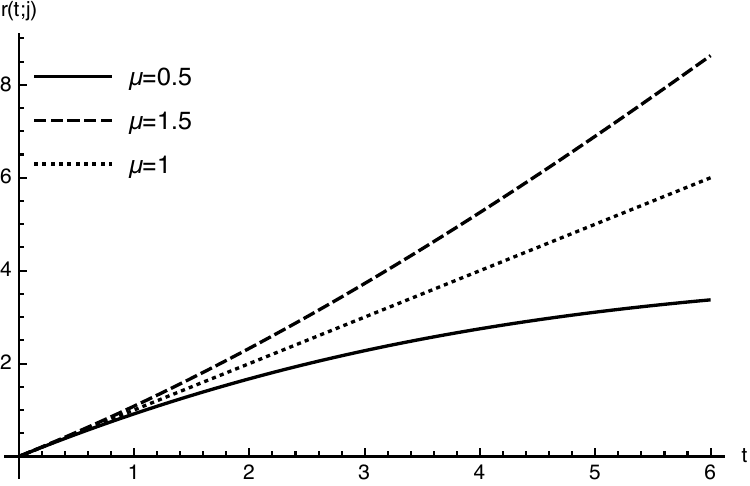}}\qquad
				\subfigure[]{\includegraphics[scale=0.5]{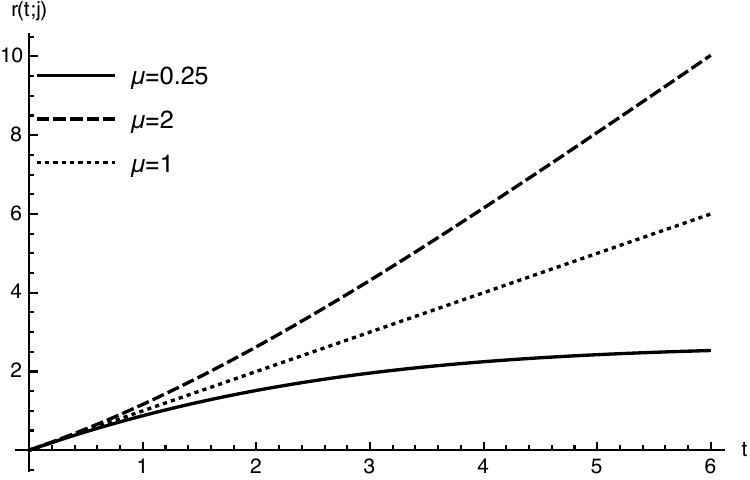}}
				\caption{The ratio $r(t;j)$ for $j=1$ and $\lambda=1$.}
				\label{fig:Figure7}
			\end{figure}
			Moreover, the following limits hold
			$$
			\begin{aligned}
				&\lim_{t\to+\infty}r(t;j)=\begin{cases}
					\displaystyle 1+\frac{\lambda+\mu}{j(\lambda-\mu)},\qquad &\lambda>\mu\\
					+\infty,\qquad &\lambda\le \mu,
				\end{cases}\qquad 
				&\lim_{t\to 0} r(t;j)=1-\frac{1}{j},\qquad \lim_{j\to +\infty} r(t;j)=1.
			\end{aligned}
			$$
			Note that 
			$$
			r(t;1)=2\lambda\left(\frac{1}{\lambda-\mu}-\frac{t}{e^{(\lambda-\mu)t}-1}\right),\qquad t\ge 0.
			$$

			We note that initially the processes $X(t)$ and $Y(t)$ are negatively correlated, for $j>1$ since $r(t;j)$ for $t\to 0$ is smaller than 1. Moreover, since $\lim_{t\to 0} r(t;j)\le1<\lim_{t\to+\infty} r(t;j)$ and the function $f(t):=-\frac{\lambda+\mu}{\lambda-\mu}+\frac{2\lambda t}{e^{(\lambda-\mu)t}-1}$ is decreasing for any $t\ge 0$, owing to \eqref{rCostanti}, there exists a single time instant $t^*$ in which  $r(t^*;j)=1$ or, equivalently, $Corr(t^*;j)=0$.
			We remark that the correlation of $X(t)$ and $Y(t)$ does not depend on $j$, as we have already pointed out, whereas the index $r(t;j)$ depens on $j$ and for great values of $j$ it is near 1.
			
	\section{Proportional intensities}\label{Section9}
	In this section, we consider the case of proportional and time-dependent spreader and inactivity rates, i.e.\ $$\lambda(t)=\rho\mu(t),\qquad \rho>0,\quad t\ge 0.$$
	In the context related to the growth of biological populations, the assumption that the birth and death rates are linked by a proportionality coefficient $\rho>0$ reflects an interaction between population growth and the regulation of its dynamics. A self-regulation system of this kind can also find an analogue in the context of spreading of fake news. In particular, there are mechanisms and processes through which the dissemination of fake news can be self-regulated over time through, for instance, information filtering by recipients (i.e., the recipient's ability to discern between accurate and fake news), technological interventions (such as the implementation of algorithms and technological tools to identify and limit the spread of fake news), education, and awareness (the greater the public's awareness of the risks and effects of fake news, the greater their propensity to exclusively disseminate accurate news). 
	\par
	Also in this case, it is possible to use the method of characteristics. The corresponding system of ordinary differential equations is given by \eqref{PDEsG-const}, with $P(z_1)=-\mu(t)\left[z_1^2\rho-z_1(\rho+1)+z_2\right]$.	
	With a reasoning similar to that of Section 5, we get the following expression for the p.g.f.
$$
G_j(z_1,z_2,t)=\begin{cases}
	\displaystyle \left(\frac{\xi_2(z_1-\xi_1)-\xi_1(z_1-\xi_2)\exp\left((\xi_2-\xi_1)M(t)\right)}{z_1-\xi_1-(z_1-\xi_2)\exp\left((\xi_2-\xi_1)M(t)\right)}\right)^j,\; &\rho\neq 1,\\
	\displaystyle \left(\frac{z_1\left(e^{2M(t) \sqrt{1-z_2}}(1-\sqrt{1-z_2})-1-\sqrt{1-z_2}\right)+z_2\left(1-e^{2M(t) \sqrt{1-z_2}}\right)}{1-z_1-e^{2M(t) \sqrt{1-z_2}}(1-z_1+\sqrt{1-z_2})-\sqrt{1-z_2}}\right)^j,\; &\rho= 1,
\end{cases}
$$
where 
\begin{equation} \label{intmu}
	M(t):=\int_0^t\mu(\tau)\mathrm{d}\tau,\qquad t\ge 0,
\end{equation}
represents the cumulative inactivity intensity, $|z_1|\le1$, $|z_2|\le1$, $t\ge 0$, and
$$
\xi_{1,2}:=\xi_{1,2}(z_2)=\frac{\rho+1\mp\sqrt{(\rho+1)^2-4z_2\rho}}{2\rho},\qquad \xi_1<\xi_2.
$$
Note that $G_j(1,1,t)=1$, as expected. Also in this case, as in Section 5, it is hard to obtain the probabilities $p_{n,k}(t)$ from $G_j(z_1,z_2,t)$.
Recalling Eq.\ \eqref{Pa}, in this case we get
\begin{equation}\label{probPAprop}
	p_{A_j}(t)=\begin{cases}
		\displaystyle \left(\frac{e^{\frac{\rho-1}{\rho}M(t)}-1}{\rho e^{\frac{\rho-1}{\rho}M(t)}-1}\right)^j,\qquad &\rho\neq 1\\
		\displaystyle \left(\frac{M(t)}{1+M(t)}\right)^j,\qquad &\rho=1.
	\end{cases}
\end{equation}
Note that the probabilities $p_{A_j}(t)$ given in Eq.\ \eqref{probPAcost} and in Eq.\ \eqref{probPAprop} coincide with the probabilities of extinction of a birth-death process given, for example, in Eq.\ (4.41) and Eq.\ (7.42a) of \cite{Ricciardi}.

Note that, since $M(t)$ is an increasing and positive function, $\lim_{t\to +\infty}M(t)\in(0,+\infty]$. Hence, when $\lim_{t\to +\infty}M(t)=M\in\mathbb R_+$, we have
$$
\lim_{t\to +\infty}p_{A_j}(t)=\begin{cases}
	\displaystyle \left(\frac{e^{\frac{\rho-1}{\rho}M}-1}{\rho e^{\frac{\rho-1}{\rho}M}-1}\right)^j,\qquad &\rho\neq 1\\
	\displaystyle \left(\frac{M}{1+M}\right)^j,\qquad &\rho=1.
\end{cases}
$$
Otherwise, when $\lim_{t\to +\infty}M(t)=+\infty$, we get
$$
\lim_{t\to+\infty}p_{A_j}(t)=\begin{cases}
	\displaystyle \left(\frac{1}{\rho}\right)^j, \qquad &\rho>1\\
	1,\qquad &0<\rho\le 1.
\end{cases}.
$$
Note that when $\lim_{t\to +\infty}M(t)=+\infty$, with a reasoning similar to that used in Section 5, it is possible to show that
$$
\lim_{t\to +\infty}p_{0,k}(t)=\left(1+\frac{1}{\rho}\right)^j\left(\frac{\rho}{(1+\rho)^2}\right)^k\left[\binom{2k-j-1}{k-1}-\binom{2k-j-1}{k}\right].
$$
Also in this case, $\lim_{t\to +\infty}\sum_{k=j}^{+\infty}p_{0,k}(t)=\lim_{t\to+\infty}p_{A_j}(t)$ as expected.
\begin{example}
	The probability $p_{0,1}(t)$, i.e. the probability that the process $(X(t), Y(t))$ starting from $(0,1)$ is then absorbed in $(1,0)$ by time $t$, is given by
	$$
	p_{0,1}(t)=\left.\frac{\partial}{\partial z_2}G_j(z_1,z_2,t)\right|_{z_1,z_2\to 0}=\begin{cases}
		\displaystyle \frac{1-e^{-\frac{1+\rho}{\rho}M(t)}}{1+\rho},\qquad &\rho\neq 1\\
		\displaystyle\frac{1-e^{-2M(t)}}{2},\qquad &\rho=1,
	\end{cases}
	$$
	where $M(t)$ is defined in \eqref{intmu}.
\end{example}
\subsection{The conditional moments of $X(t)$}
Under the assumptions $\lambda(t)=\rho\mu(t)$ with $\rho>0$, the conditional mean and the conditional variance of $X(t)$ are given by
$$
m_X(t;j)=
\displaystyle j e^{(\rho-1)M(t)},\qquad t\ge 0
$$
and
\begin{equation}\label{VarXprop}
	Var_X(t;j)=
	\begin{cases} 
		\displaystyle \frac{j(\rho+1)}{\rho-1}e^{(\rho-1)M(t)}\cdot\left[e^{(\rho-1)M(t)}-1\right],\qquad &\rho\neq 1\\
		\displaystyle 2jM(t),\qquad &\rho=1,
	\end{cases}
\end{equation}
for any $t\ge 0$. Moreover, the Fano factor is given by
$$
D_X(t;j)=
\begin{cases}
	\displaystyle \frac{\rho+1}{\rho-1}\left[e^{(\rho-1)M(t)}-1\right],\qquad &\rho\neq 1\\
	\displaystyle 2M(t),\qquad &\rho=1,
\end{cases}
$$
and the coefficient of variation by
$$
\sigma_X(t;j)=
\begin{cases}
	\displaystyle \frac{e^{-\rho/2M(t)}}{j}\sqrt{\frac{(\rho+1)j}{\rho-1}\left[e^{\rho M(t)}-e^{M(t)}\right]},\qquad &\rho\neq 1\\
	\displaystyle \sqrt{\frac{2M(t)}{j}},\qquad &\rho=1.
\end{cases}
$$
In Figure \ref{fig:Figure2} we provide some plots of $m_X(t;j)$ and $Var_X(t;j)$ by considering $\mu(t)=\mu+\alpha\cos\left(\frac{2\pi}{Q}t\right)$ with $Q>0$, $t\ge 0$, $\mu>|\alpha|>0$. Such choice may represent the periodical fluctuation of the inactivity rate owing to, for example, periodic denies of the rumor.
\begin{figure}[t]
	\centering
	\subfigure[]{\includegraphics[scale=0.5]{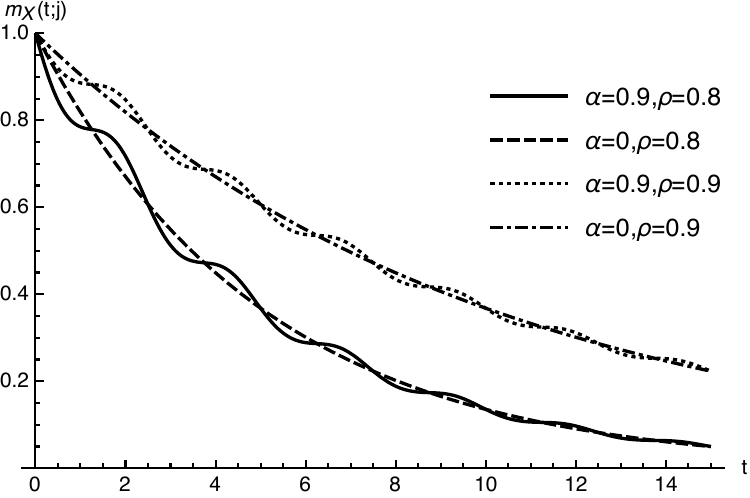}}\qquad
	\subfigure[]{\includegraphics[scale=0.5]{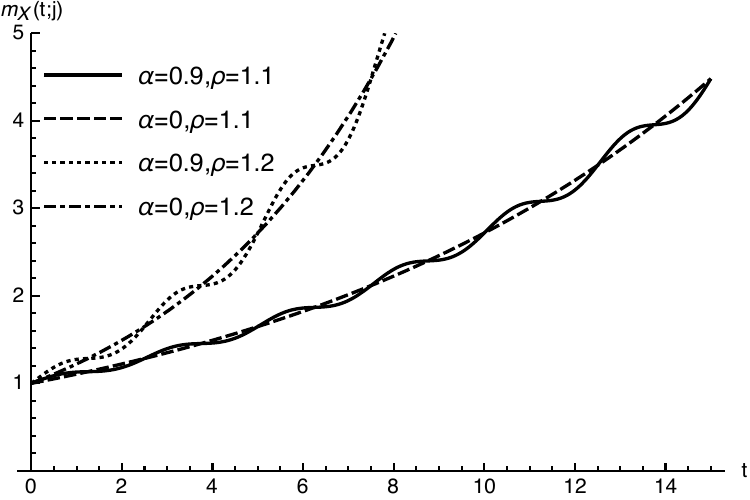}}\\
	\subfigure[]{\includegraphics[scale=0.5]{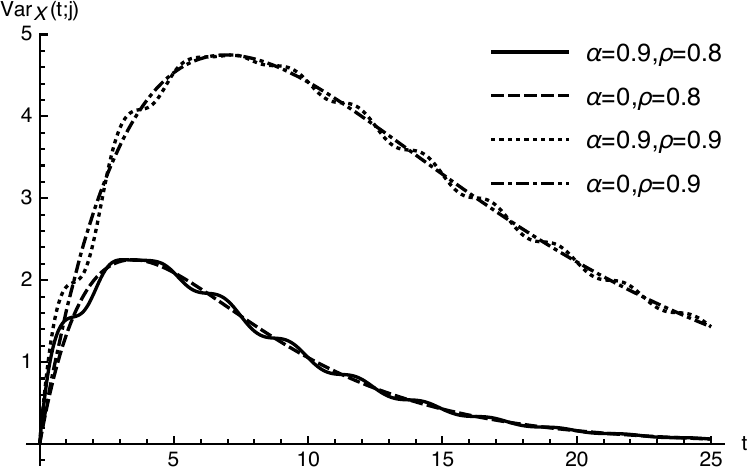}}\qquad
	\subfigure[]{\includegraphics[scale=0.5]{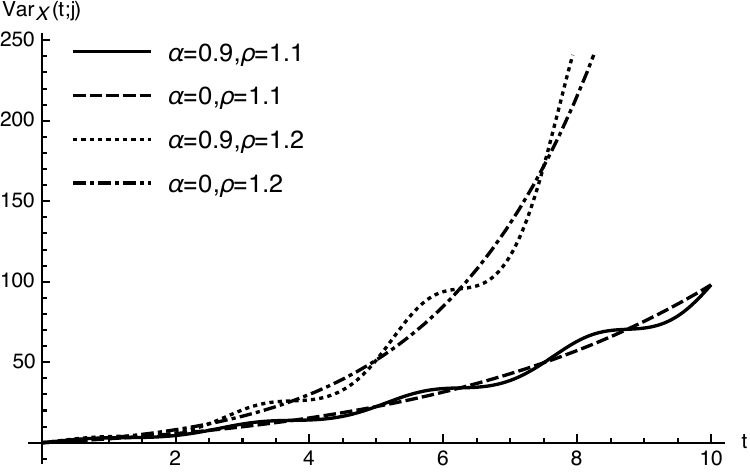}}
	\caption{The conditional mean $m_X(t;j)$ and the conditional variance $Var_X(t;j)$ for $\lambda=\rho\mu(t)$ and $\mu(t)=\mu+\alpha\cos\left(\frac{2\pi}{Q}t\right)$ with $\mu=j=1$ and $Q=2.5$.}
	\label{fig:Figure2}
\end{figure}
Note that, considering that $\lim_{t\to +\infty}M(t)\in[0,+\infty]$, when $\lim_{t\to+\infty} M(t)=M\in\mathbb R$ then $m_X(t;j)$, $Var_X(t;j)$, $D_X(t;j)$ and $\sigma_X(t;j)$ converge, otherwise when $\lim_{t\to+\infty}M(t)=+\infty$ the following limits hold
$$
\begin{aligned}
	&\lim_{t\to+\infty} m_X(t;j)=\begin{cases}
		0,\qquad &\rho<1\\
		j,\qquad &\rho=1\\
		+\infty,\qquad &\rho>1
	\end{cases},
\qquad
\lim_{t\to+\infty}Var_X(t;j)=\begin{cases}
	0,\qquad &\rho<1\\
	+\infty,\qquad &\rho\ge 1
\end{cases},\\
&\lim_{t\to+\infty}D_X(t;j)=\begin{cases}
	\frac{\rho+1}{1-\rho},\qquad &\rho<1\\
	+\infty\qquad &\rho\ge 1
\end{cases},
\qquad 
\lim_{t\to+\infty}\sigma_X(t;j)=\begin{cases}
+\infty,\qquad &\rho\le1\\
	\sqrt{\frac{\rho+1}{j(\rho-1)}},\qquad  &\rho> 1.
\end{cases}
\end{aligned}
$$
\subsection{The conditional moments of $Y(t)$}
In the case of proportional intensities, the conditional expected value of $Y(t)$ is given by
\begin{equation}\label{mYprop}
	m_Y(t;j)=
	\begin{cases}
		\displaystyle	\frac{j}{\rho-1}\left(e^{(\rho-1)M(t)}-1\right),\qquad &\rho\neq 1\\
		\displaystyle	jM(t),\qquad &\rho=1,
	\end{cases}
\end{equation}
where $M(t)$ is defined in Eq.\ \eqref{intmu} for any $t\ge 0$. Considering that the sign of $\mu(t)$ is always positive, the conditional expected value $m_Y(t;j)$ is increasing. The concavity of $m_Y(t;j)$ depends on the sign of $\rho-1$ and on the monotonicity of $\mu(t)$. Indeed, since
$$
\frac{d^2}{dt^2}m_Y(t;j)=j e^{(\rho-1)M(t)}
\left(\frac{d}{dt}\mu(t)+\mu^2(t)(\rho-1)\right), \qquad t\ge 0,
$$
then the concavity of $m_Y(t;j)$ is upward when $\mu(t)$ is an increasing function and $\rho>1$, whereas it is downward when $\mu(t)$ is a decreasing function and $0<\rho<1$.
In Table \ref{tab:Tabella2} the sign of the difference between $m_X(t;j)$ and $m_Y(t;j)$ is analyzed. Note that when the spreader intensity $\lambda(t)$ is at least equal to the double of the inactivity intensity $\mu(t)$, then $m_X(t;j)$ is greater than $m_Y(t;j)$ for any $t\ge 0$, similarly to the result of Section 5.
\begin{table}[t]
	\caption{The sign of the difference between the conditional means for different range of values of $\rho$.}
	\label{tab:Tabella2}
	\centering
	\begin{tabular}{l|l} 
		&                                                                                                                                                                                                                                                                           \\ 
		\hline
		$0<\rho<1$                   & $\displaystyle m_X(t;j)>m_Y(t;j) \iff M(t)>\frac{1}{\rho-1}\log\left(\frac{1}{2-\rho}\right)$  \\
		$\rho=1$                   	& $\displaystyle m_X(t;j)>m_Y(t;j)\iff M(t)<1$                                                                                                                                                                                       \\
		$1<\rho<2$ 			  & $\displaystyle m_X(t;j)>m_Y(t;j)\iff M(t)<\frac{1}{\rho-1}\log\left(\frac{1}{2-\rho}\right)$  \\
		$\rho\ge2$ 				& $\displaystyle m_X(t;j)>m_Y(t;j)\qquad \forall t\ge 0$.                                                                                                                                                                                
	\end{tabular}
\end{table}

The conditional second order moment of $Y(t)$ is given by
\begin{equation}\label{m2Yprop}
	\begin{aligned}
		m_{2,Y}(t;j)=
		\begin{cases}
			\displaystyle	\frac{j(\rho+1-2j)}{(\rho-1)^2}\left(e^{(\rho-1)M(t)}-1\right)-\frac{4\rho j }{(\rho-1)^2}M(t)\\
			\displaystyle	\times e^{(\rho-1)M(t)}+\frac{j(1+j(\rho-1)+\rho)}{(\rho-1)^3}\left(e^{2(\rho-1)M(t)}-1\right) , &\rho\neq 1\\
			\displaystyle\frac{j}{3}M(t) \left(3+3(j-1)M(t)+ 2M^2(t)\right),\qquad &\rho=1,
		\end{cases}
	\end{aligned}
\end{equation}
where $M(t)$ is defined in Eq.\ \eqref{intmu} for any $t\ge 0$. The expressions of the variance and of the other indexes are also available in closed form but they are omitted for brevity.
We remark that, if $\lim_{t\to+\infty}M(t)=M\in\mathbb R_+$, then $m_Y(t;j)$, $Var_Y(t;j)$, $D_Y(t;j)$ and $\sigma_Y(t;j)$ are finite. Otherwise, if $\lim_{t\to+\infty}M(t)=+\infty$, then the following limits hold
\begin{equation*}
	\lim_{t\to+\infty} m_Y(t;j)=\begin{cases}
		\frac{j}{1-\rho},\quad &\rho<1\\
		+\infty, \quad &\rho\ge 1
	\end{cases},
\quad \lim_{t\to+\infty}m_{2,Y}(t;j)=\begin{cases}
\displaystyle	\frac{j(j(\rho-1)+(3-5\rho)\rho)}{(\rho-1)^3},\quad &\rho<1\\
	+\infty,\quad &\rho\ge 1.
\end{cases}
\end{equation*}

\subsection{Correlation between $X(t)$ and $Y(t)$}
The conditional expected value of $X(t)\cdot Y(t)$ is given by
$$
\begin{aligned}
	&m_{X\cdot Y}(t;j)=\begin{cases}
		\displaystyle je^{(\rho-1)M(t)}\left[\frac{2\rho}{1-\rho}M(t)+\left(j-1-\frac{2\rho}{1-\rho}\right)\frac{e^{(\rho-1)M(t)}-1}{\rho-1}\right],\quad &\rho\neq 1\\
		\displaystyle jM(t)\left(j-1+M(t)\right),\quad &\rho=1,
	\end{cases}
\end{aligned}
$$
where $M(t)$ is defined in Eq.\ \eqref{intmu} for any $t>0$. {In this case, the study regarding the covariance $Cov(t;j)$ and the correlation $Corr(t;j)$ is omitted since such indexes possess cumbersome expressions. 
Hence, we focus on the determination and analysis of the adimensional correlation index $r(t;j)$.} 
Recalling Definition \ref{coeffr}, the coefficient $r(t;j)$ is given by
$$
r(t;j)=\begin{cases}
	\displaystyle 
	1+\frac{1}{j}+\frac{2}{\rho-1}-\frac{2M(t)\rho}{j(e^{(\rho-1)M(t)}-1)},\qquad &\rho\neq1\vspace{0.3cm}\\
	\displaystyle 1+\frac{M(t)-1}{j},\qquad &\rho=1,
\end{cases}
$$
where $M(t)$ is defined in Eq.\ \eqref{intmu} for any $t\ge 0$. In this case $\lim_{j\to+\infty}r(t;j)=1$, meaning that when the number of initial spreaders is high, then the processes $X(t)$ and $Y(t)$ are uncorrelated.
We note that, when $\lim_{t\to+\infty} M(t)=M\in\mathbb R$, both the mean $m_{X\cdot Y}(t)$ and the adimensional correlation index $r(t;j)$ are finite. Otherwise, when $\lim_{t\to+\infty} M(t)=+\infty$, the following limits hold
$$
\lim_{t\to+\infty} m_{X\cdot Y}(t;j)=\begin{cases}
	0, \qquad &\rho<1\\
	+\infty,\qquad &\rho\ge 1,
\end{cases}
\lim_{t\to+\infty}r(t;j)=\begin{cases}
	+\infty,\qquad &\rho\le 1\\
	1+\frac1j+\frac{2}{\rho-1},\qquad &\rho>1.
\end{cases}
$$
Moreover, if $0<\rho\le 1$ then $r(0,j)<1$, otherwise $r(0,j)>1$, meaning that if $0<\rho\le 1$, the processes are negatively correlated at the beginning of the time-evolution. This result can be interpreted from a modeling point of view as follows. When spreaders grow more slowly than inactives (i.e. $0<\rho\le 1$), then all individuals behave as skeptical from the beginning  and they therefore do not intend to spread the news. Similarly, when spreaders grow faster than inactive ones (i.e. $\rho>1$), then the population initially driven to spread the news being influenced by the rapid growth of the number of spreaders. 
Some specific cases are discussed in the following section. 
	\section{Special cases of the conditional means}\label{Sec:special-cases}
	In this section, we analyze some special situations in which the spreader and the inactivity intensities are proportional, i.e. $\lambda(t)=\rho\mu(t)$, $\rho>0$ (as in Section 6) and the conditional means $m_X(t;j)$ and $m_Y(t;j)$ are set equal to some very well-known growth curves. 
	Considering the application to the context of fake news, it is expected a non-explosive behavior from the category of spreaders and consequently from the inactives. Indeed, it usually happens that after an initial period of explosion in the spread of a fake news, it gradually tends to disappear, thus reaching only a limited portion of the population. Therefore, it is plausible to expect an average trend for spreaders and inactives of sigmoidal type. 
	This assumption is also appropriate for taking into account that the members of the population are generally not particularly inclined to follow fake news irrationally, although (as detailed in Section \ref{Section2}) the total size $N$ of the population is seen as very large.	
	In addition, it is worth mentioning that even if the present model can be applied, in general, not only for the diffusion of fake news but for any kind of news, 
the specific assumption that the conditional means of the components follow a sigmoidal pattern is appropriate specifically in the context of the dissemination of fake news which, by their nature, are not intended to achieve widespread diffusion in the population. 
	Hence, it seems reasonable to give attention to the cases when the means $m_X(t;j)$ and $m_Y(t;j)$ are equal to specific sigmoidal growth functions, which are often used to model real growth phenomena (for a deep study of sigmoidal models see, for instance, Albano {\em et al}.\ \cite{Albanoetal}). Regarding the process $Y(t)$, it is known that $Y(0)=0$, hence we investigate the conditions under which $m_Y(t;j)$ is equal to sigmoidal growth curves starting from $0$. {It seems reasonable to assume that the initial size of the inactive individuals is equal to 0, since at the beginning of the diffusion all the individuals believe in the rumor and have the intention of spreading it. Likewise, we assume that the initial size of spreaders is positive (i.e. $X(0)=j>0$); otherwise, the rumor diffusion would not ensue.}
\par
	 In particular, for the process $X(t)$ we consider:
	 \begin{itemize}
	 	\item[(i)] The classical Gompertz growth curve 	(cf.\ Rom\'an-Rom\'an \textit{et al.}\ \cite{Romaneral2019}), which has been used in several contexts during recent years, especially related to tumor growth (see for example Rom\'an-Rom\'an \textit{et al.}\ \cite{Romanetal2021}).
	 	\item[(ii)] The generalized Gompertz growth curve introduced in Asadi \textit{et al.}\ \cite{DiCrescenzoetal2023}, which is related to the generalized Pareto distribution. In particular, such growth function can have different kinds of behaviour (it can be sigmoidal, increasing with infinite limit value, bell-shaped) depending on the values of the parameters. For our interest, we consider only the case in which the generalized Gompertz curve has a sigmoidal shape.
	 	\item[(iii)] The classical logistic growth curve, which is one of the most popular growth models characterized by a finite carrying capacity. It was originally derived by Verhulst to describe the self-limiting growth of a biological population and during the years it has been deeply analyzed (see, for example, Di Crescenzo and Paraggio \cite{DiCrescenzoParaggio2019}).
	 	\item[(iv)] The extended logistic growth curve (cf.\ Di Crescenzo \textit{et al.} \cite{DiCrescenzoetal2023a} and San MartÃ¬n \cite{SanMartin2020}), which has been proposed to model the rumor spread among restricted populations.
	 	\item[(v)] The multisigmoidal logistic growth curve (studied in Di Crescenzo \textit{et al.} \ \cite{DiCrescenzoetal2021} and \cite{DiCrescenzoetal2022a}),  which exhibits multiple inflections.
	 	\item[(vi)] The modified Korf growth curve (analyzed in Di Crescenzo and Spina \cite{DiCrescenzoSpina2016}), that captures certain features of both the Gompertz and Korf growth laws.
	 \end{itemize} 
 Whereas, for the process $Y(t)$ we consider:
 \begin{itemize}
 	\item[(I)] The classical Korf model (cf.\ Eq.\ (9) of Di Crescenzo and Spina \cite{DiCrescenzoSpina2016}), which starts from the state $0$  at time $0$, then evolving as a sigmoidal function.	
 	\item[(II)] The Mitscherlich model (cf.\ Eq.\ (2) of Ware \textit{et al}.\ \cite{Wareetal1982}) which starts from the state $0$ at time $0$ and has a finite limit. Such growth curve has been proposed to model the plant growth and to determine critical nutrient deficiency levels.
 \end{itemize} 
 Note that all the aforementioned curves are S-shaped, namely they start with an explosion of exponential type, then an inflection (or more) occurs and finally they flattens up to a constant value, known as carrying capacity. 

	Table \ref{tab:Tabella7} summarizes the most relevant information considering  cases (i)$\div$(vi), such as the expression of the conditional mean $m_X(t;j)$, which corresponds to the expression of the corresponding growth curve, the expression of the cumulative inactivity intensity $M(t)$, the time intervals in which the sign of the difference $m_X(t;j)-m_Y(t;j)$ is positive, the sign of the correlation coefficient $Corr(t;j)$, the Fano factor, the time intervals of overdispersion and the coefficient of variation for both the conditional processes $X(t)$ and $Y(t)$ given the initial state $(X(0), Y(0))=(j,0)$. {To shorten the formulas contained in Table \ref{tab:Tabella7}, for case (iv), we set}
	$$
	m_X^{(el)}(t;j):=\frac{(\varepsilon-1)N(N-j)+e^{(1+\varepsilon)t}N(2\varepsilon j+N-\varepsilon N)}{2\varepsilon (N-j)+e^{(1+\varepsilon)t}(2\varepsilon j+N-\varepsilon N)},\qquad t\ge 0,
	$$
	and
	$$
	M^{(el)}(t):=\frac{1}{\rho-1}\log\left(\frac{(\varepsilon-1)N(N-j)+e^{(1+\varepsilon)t}N(N(\varepsilon-1)-2\varepsilon j)}{2\varepsilon j (N-j)+e^{(1+\varepsilon)t}j (2\varepsilon j +N-\varepsilon N)}\right),\qquad t\ge 0.
	$$
		
	\begin{sidewaystable}[]
		\caption{Some relevant characteristic of the conditional processes $X(t)$ and $Y(t)$ given the initial state $(X(0), Y(0))=(j,0)$ by considering different growth models for $m_X(t;j)$. Note that $\alpha,\beta,A,b,C,j>0$, $\rho>1$. For the logistic case one has $0<j<C$, for the multisigmoidal logistic $Q_\beta(t):=\sum_{k=1}^n \beta_kt^k$, with $\beta_n<0$.}
		\tiny
		\centering 
		\begin{tabular}{l|llll}
			growth curve            & $m_X(t;j)$                                                                                                                                                       & $M(t)$                                                   & $\sgn (m_X(t;j)-m_Y(t;j))=+1,\;t\ge \widetilde t$                                                                                                                       & $\sgn (Corr(t;j))$  \\     \hline
			(i) Gompertz                & $j\exp\left(\alpha(1-e^{-\beta t})\right)$                                                                                                                        & $\frac{\alpha\left(1-e^{-\beta t}\right)}{\rho-1}$                   & $\widetilde t=\begin{cases} -\frac{1}{\beta}\left(1-\frac 1 \alpha\log(\rho-2)\right), &\text{ if } 1<\rho<2\\ 0, &\text{ otherwise}\end{cases}$           & $+1,t\ge 0$     \\ \hline
			(ii) generalized Gompertz    & $j\exp\left(Ab\left(1-\frac{b}{t+b}\right)\right)$                                                                                                                      & $\frac{Abt}{(\rho-1)(t+b)}$            & $\widetilde t=\begin{cases} 0,&\text{ if } \rho>2\\ +\infty &\text{ otherwise} \end{cases}$                                                                                                                                   & changing                               \\ \hline
			(iii) logistic                & $\frac{Cj}{j+(C-j)e^{-rt}}$                                                                                                                                      & $\frac{1}{\rho-1}\left({rt+\log\frac{C}{C+j(e^{rt}-1)}}\right)$          & $\widetilde t =\begin{cases} \frac 1r\log\left(\frac{C-j}{C(2-\rho)-j}\right),&\text{ if } 1<\rho<2, C>\frac{j}{2-\rho},\\ 0,&\text{ otherwise}\end{cases}$ & changing          \\ \hline
			(iv) extended logistic       & $m_X^{(el)}(t;j)$ & $M^{(el)}(t)$                                                                  & NA in closed form                                                                                                                                                     & changing       \\ \hline
			(v) multisigmoidal logistic & $\frac{Cj}{j+(C-j)e^{Q_\beta(t)}}$                                                                                                                               & $\frac{1}{\rho-1}\left({Q_\beta(t)+\log\frac{C}{C+j(e^{Q_\beta(t)}-1)}}\right)$  & NA in closed form                                                                                                                                                     & changing    \\ \hline
			(vi) modified Korf           & $j\exp\left[\frac\alpha\beta\left(1-(1+t)^{-\beta}\right)\right]$                                                                                                & $\frac{A\left(1-(1+t)^{-b}\right)}{b(\rho-1)}$                & $\widetilde t=\begin{cases} e^{-1/\beta}\left(1+\frac\beta\alpha\log(2-\rho)\right), &\text{ if }  \alpha+\beta\log(2-\rho)>0 \\ 0, &\text{ otherwise}\end{cases}$      & changing                       
		\end{tabular}\\
		
		\vspace{0.5cm}
		
		\begin{tabular}{l|llll}
			growth curve                             &process & \multicolumn{1}{c}{Fano factor $D_Z(t;j)$} & \multicolumn{1}{c}{underdispersion conditions} & \multicolumn{1}{c}{coefficient of variation $\sigma_Z(t;j)$} \\ \hline
			\multirow{2}{*}{(i) Gompertz}                & $X$                                     &$\frac{\rho+1}{\rho-1}\left(\exp\left(\alpha-\alpha e^{-\beta t}\right)-1\right)$               &         $t\ge \widetilde t= \begin{cases} \frac1\beta\log\frac{\alpha}{\alpha-\log2+\log{1+\rho}} &\text{ if } \alpha>\log\frac{2}{1+\rho}\\
				0 &\text{ otherwise }\end{cases}$                                            &   $\frac{e^{-\alpha/2}}{j}\sqrt{\frac{j(\rho+1)}{\rho-1}\left(e^{\alpha}-\exp\left(\alpha e^{-\beta t}\right)\right)}$                                                           \\
			& $Y$ &        omitted                                  &     NA in closed form                                                         &  omitted                                                              \\ \hline
			\multirow{2}{*}{(ii) generalized Gompertz}    & $X$ &       $\frac{\rho+1}{\rho-1}\left(\exp\frac{Abt}{t+b}-1\right)$                                     &       $t\le \widetilde t=\begin{cases}\frac{b\log\frac{2\rho}{\rho+1}}{\log\frac{2\rho}{\rho+1}-Ab},\text{ if } Ab-\log\frac{2\rho}{\rho+1}<0\\
				+\infty\text{ otherwise }\end{cases}$                                                        		
			&  $ \sqrt{\frac{j(\rho+1)}{\rho-1}\left(1-\exp\frac{Abt}{b+t}\right)}$                                 \\
			& $Y$ &         omitted                                   &      NA  in closed form                                                      &                            omitted                                  \\ \hline
			\multirow{2}{*}{(iii) logistic}                & $X$ &      $\frac{\rho+1}{\rho-1}\left(\frac{C}{j+e^{-rt}(C-j)}-1\right)$                                      &  $t<\frac{1}{r}\log\left(\frac{(C-j)(\rho-1)}{\rho(C-j)+C+j}\right)$                                                             &       $\frac{e^{-rt}\left(C+(e^{rt}-1)j\right)}{Cj}\sqrt{\frac{Ce^{rt}(e^{rt}-1)(\rho+1)(C-j)j}{(\rho-1)\left(C+j(e^{rt}-1)\right)^2}}$                                                       \\
			& $Y$ &   omitted                                         &    NA in closed form                                                           &                      omitted                                        \\ \hline
			\multirow{2}{*}{(iv) extended logistic}       & $X$ & $ \frac{(e^{(1+\varepsilon)t}-1)(N-j)(\rho+1)(2\varepsilon j +N(1-\varepsilon))}{j(2\varepsilon (N-j)+e^{(1+\varepsilon)t}(2\varepsilon j +N(1-\varepsilon)))(\rho-1)}    $                                      &   NA in closed form                                                           &   omitted                                                           \\
			& $Y$ &   omitted                                         &       NA  in closed form                                                   &  omitted                                                            \\ \hline
			\multirow{2}{*}{(v) multisigmoidal logistic} & $X$ &  $\frac{\rho+1}{\rho-1}\left(\frac{C}{j+e^{Q_\beta(t)}(C-j)}-1\right)$                                          &   NA in closed form                                                            &    $\frac{e^{Q_\beta(t)}\left(C+(e^{-Q_\beta(t)}-1)j\right)}{Cj}\sqrt{\frac{Ce^{-Q_\beta(t)}(e^{-Q_\beta(t)}-1)(\rho+1)(C-j)j}{(\rho-1)\left(C+j(e^{-Q_\beta(t)}-1)\right)^2}}$                                                            \\
			& $Y$ &      omitted                                      &        NA in closed form                                                       &   omitted                                                           \\ \hline
			\multirow{2}{*}{(vi) modified Korf}           & $X$ &   $\frac{\rho+1}{\rho-1}\left(\exp\left\{\frac{\alpha}{\beta}\left[1-(1+t)^{-\beta}\right]\right\}-1\right),$                                         &  $t\ge\widetilde t=\begin{cases}
				0,\text{ if } \alpha-\beta\log\left(\frac{2\rho}{\rho+1}\right)<0\\
				\left[1-\frac\beta\alpha\log\left(\frac{2\rho}{1+\rho}\right)\right]^{-1/\beta}-1, \text{ otherwise}
			\end{cases}    $                                                         &         $\sqrt{\frac{1+\rho}{j(\rho-1)}\left[1-\exp\left\{\frac\alpha\beta\left[(1+t)^{-\beta}-1\right]\right\}\right]}$                                                     \\
			& $Y$ &    omitted                                        &    NA  in closed form                                                         &      omitted 
		\end{tabular}
		\label{tab:Tabella7}
	\end{sidewaystable}

	Note that knowing $M(t)$ gives the possibility to get $Var_X(t;j)$, $m_Y(t;j)$ and $Var_Y(t;j)$ by means of Eqs.\ \eqref{VarXprop}, \eqref{mYprop} and \eqref{m2Yprop}, respectively. Moreover, we remark that when $m_X(t;j)-m_Y(t;j)>0$ then the expected number of spreaders is greater than the expected number of inactive individuals.
	Other interesting limit behaviours are given in Table \ref{tab:Tabella8} where the function $\mathcal V_{i,j}$, $\mathcal C_{i,j}$ and $\mathcal R_{i,j}$, for $i=1,2,\dots,6$, $j=1,2,3$  are given in Table \ref{tab:Tabella10}.
	We point out that in the Gompertz, generalized Gompertz, extended logistic, modified Korf cases, all the available limits with respect to the time are finite, whereas in the logistic case the variance $Var_Y(t;j)$ diverges as $t\to+\infty$.
	In Figures \ref{fig:Figure23}(a)$\div$(d), we provide the conditional means $m_X(t;j)$ and $m_Y(t;j)$, the conditional variance $Var_X(t;j)$ and $Var_Y(t;j)$. Moreover, Figures \ref{fig:Figure23}(e)-(f) show the conditional correlation coefficient $Corr(t;j)$ and the conditional coefficient $r(t;j)$. Note that in the Gompertz case, the processes $X(t)$ and $Y(t)$ given the initial state $(X(0), Y(0))=(j,0)$ are always positively correlated, whereas in the logistic case the correlation coefficient is close to 0. 
	{
	Comparing the results related to cases (i)-(vi) shown in Figure  \ref{fig:Figure23}, we remark that expected number of spreaders and the expected number of inactive individuals tend to the corresponding carrying capacities (i.e. the maximum reachable number of spreaders/inactives in the population) but with different speed. The slowest case is the generalized Gompertz, whereas the fastest is the multisigmoidal logistic. Hence, in the multisigmoidal logistic case (generalized Gompertz case), the fake news exhibits the fastest (the slowest) diffusion among the population and the fastest (slowest) growth of the inactives. Moreover, the correlation between spreaders and inactives is the highest (from the beginning to the end of the diffusion) in the Gompertz case, whereas it is the lowest in the multisigmoidal logistic case at the beginning of the propagation of the rumor and in the logistic case for large times. Similarly, the adimensional correlation index \eqref{corr-index} is always greater than 1 (corresponding to highly positive correlation) in the Gompertz case whereas in all the other cases  such index  is smaller than 1 (corresponding to negative correlation) at the beginning, and it  is greater than 1 for large times.
}

	Moreover, Table \ref{tab:Tabella11} summarizes the most relevant information considering the case in which $m_Y(t;j)$ is set equal to the Korf and to the Mitscherlich growth function, such as the expression of the conditional mean $m_Y(t;j)$, the expression of the cumulative inactivity intensity $M(t)$, the time intervals in which the sign of the difference $m_X(t;j)-m_Y(t;j)$ is positive, the sign of the correlation coefficient $Corr(t;j)$, the Fano factor, the time intervals of overdispersion and the coefficient of variation for both the conditional processes $X(t)$ and $Y(t)$ given the initial state $(X(0), Y(0))=(j,0)$, the limit behaviors of the means, of the variances, of the covariance and of the adimensional correlation index $r(t;j)$.
{
	Note that if $t$ is large, the expected number of spreaders is greater than the expected number of inactives. Whereas, if $t$ is close to $0$ and the mean of $Y(t)$ is the Mitscherlich growth function, then the process $X(t)$  is underdispesed. On the other hand, if $m_Y(t;j)$ is the Korf function, then the process $X(t)$ is overdispersed for $t$ close to $0$. The coefficient of variation of $X(t)$ tends to a finite value for $t\to +\infty$ both in Case (I) and Case (II). Moreover, for both Case (I) and Case (II), the means $m_X(t;j)$ and $m_Y(t;j)$ are significant indexes since the corresponding variances are finite. 
Considering Figure \ref{fig:Figure25}, we note that for the specific values of parameters, the expected number of spreaders/inactives is higher in Case (I). Also for the variances of $X(t)$ and $Y(t)$, we note that in Case (I) we obtain the greatest values. Whereas, at the beginning of the propagation of the rumor, the correlation between spreaders and inactive individuals is higher in Case (I), but for large times it is higher in Case (II). Regarding the adimensional correlation index $r(t;j)$, we note that it is always higher in Case (II). Moreover, for that specified choice of parameters, one has $r(t;j)<1$ for any $t\ge 0$ in Case (I), meaning that when the expected number of inactive individuals is the Mitscherlich growth function, the components $X(t)$ and $Y(t)$ are negatively correlated. 
}
	\begin{table}[]
		\caption{Some relevant characteristic and some limit behaviors of the conditional processes $X(t)$ and $Y(t)$ given the initial state $(X(0), Y(0))=(j,0)$ by considering the Korf and the Mitscherlich growth models for $m_Y(t;j)$. Note that $\alpha,\beta,j>0$, $\rho>1$. }
		\tiny
		\centering 
		\begin{tabular}{l|llll}
			growth curve    for $Y(t)$        & $m_Y(t;j)$                                                                                                                                                       & $M(t)$                                                   & $\sgn (m_X(t;j)-m_Y(t;j))=+1,\;t\ge \widetilde t$                                                                                                                       & $\sgn (Corr(t;j))$  \\     \hline
			(I) Korf                & $\frac{j}{\rho-1}\exp\left(-\frac{\alpha}{\beta}t^{-\beta}\right)$                                                                                                                        & $\frac{1}{\rho-1}\log\left[1+\exp\left(-\frac{\alpha}{\beta}t^{-\beta}\right)\right]$                   & $\widetilde t=\begin{cases} \left(\frac{\alpha}{\beta \log\frac{2-\rho}{\rho-1}}\right)^{1/\beta}, &\text{ if } \frac{3}{2}<\rho\le 2\\ 0, &\text{ otherwise}\end{cases}$           & changing     \\  \hline
			(II) Mitscherlich  &$\beta\left(1-e^{-\alpha t}\right)$ & $\frac{1}{\rho-1}\log\left(2-e^{-\alpha t}\right)$ &$\widetilde t=\begin{cases}
				0, \; \text{if }\rho>\frac{2\beta -j}{\beta}\\
				\frac{1}{\alpha}\log\left(\frac{\beta(2-\rho)}{\beta(2-\rho)-j}\right),\; \text{otherwise}
			\end{cases}$ & changing
			
		\end{tabular}\\
		
		\vspace{0.5cm}
		
		\begin{tabular}{l|llll}
			growth curve      for $Y(t)$                       & process & \multicolumn{1}{c}{Fano factor $D_Z(t;j)$} & \multicolumn{1}{c}{underdispersion conditions} & \multicolumn{1}{c}{coefficient of variation $\sigma_Z(t;j)$} \\ \hline
			\multirow{2}{*}{(I) Korf}                & $X$                                     &$\frac{\rho+1}{\rho-1}\exp\left(-\frac{\alpha}{\beta}t^{-\beta}\right)$               &         $t\ge \widetilde t= \left(-\frac{\alpha}{\beta\log\frac{\rho-1}{\rho+1}}\right)^{1/\beta}$                                            & $\frac{1}{j(\rho-1)\sqrt{1+\exp\left(\frac\alpha\beta t^{-\beta}\right)}}$  \\
			
			& $Y$ &        omitted                                  &     NA in closed form                                                         &  omitted                                                              \\ \hline
			\multirow{2}{*}{(II) Mitscherlich}                & $X$                                     &$\frac{\rho+1}{j}\beta\left(1-e^{-\alpha t}\right)$               &         $t\le \widetilde t= \frac1\alpha\log\frac{\beta(\rho+1)}{\beta(\rho+1)-j}$                                            & $
			\sqrt{\frac{\rho+1}{j\left(\rho-1+\frac{e^{\alpha t} j}{\beta(e^{\alpha t}-1)}\right)}}$  \\
			
			& $Y$ &        omitted                                  &     NA in closed form                                                         &  omitted                                                              
			\end{tabular}
		
			\vspace{0.5cm}
			
		\begin{tabular}{l|lllllll}
		growth curve                             & \multicolumn{1}{c}{conditions}                                  & \multicolumn{1}{c}{$m_X(t;j)$} & \multicolumn{1}{c}{$m_Y(t;j)$} & \multicolumn{1}{c}{$Var_X(t;j)$} & \multicolumn{1}{c}{$Var_Y(t;j)$} & \multicolumn{1}{c}{$Cov(t;j)$} & \multicolumn{1}{c}{$r(t;j)$} \\ \hline
		\multirow{5}{*}{(I) Korf}                & $t\to+\infty$         & $2j$ 		&$\displaystyle \frac{j}{\rho-1}$			&$\displaystyle \frac{2j(\rho+1)}{\rho-1}$			&$<+\infty$		&$\frac{2j(1+\rho-\rho\log4)}{(\rho-1)^2}$ 		&$\frac{j-1+\rho+j\rho-\rho\log 4}{j(\rho-1)}$\\
		& $\alpha\to0$                                                    & $2j$				&$\frac{j}{\rho-1}$				&$\frac{2j(\rho+1)}{\rho-1}$				&$<+\infty$           &$\frac{2j(1+\rho-\rho\log4)}{(\rho-1)^2}$          &$\frac{j-1+\rho+j\rho-\rho\log 4}{j(\rho-1)}$                            \\
		& $\alpha\to+\infty$ & $j$ &$0$ &$0$ 			&$0$ 					&$0$      &$\frac{(\rho+1)(j-1)}{j(\rho-1)}$                              \\
		& $\beta\to 0$                                                     & $j$ 	&$0$			&$0$			&$0$			&$0$	&$\frac{j-1+\rho j+\rho}{j(\rho-1)}$                           \\
		& $\beta\to+\infty$                                             &  $2j$ 			&$\displaystyle \frac{j}{\rho-1}$		&$\displaystyle \frac{2j(\rho+1)}{\rho-1}$	            &$<+\infty$			&$\frac{2j(1+\rho-\rho\log4)}{(\rho-1)^2}$			&$\frac{j-1+\rho+j\rho-\rho\log 4}{j(\rho-1)}$        \\ \hline
		\multirow{5}{*}{(II) Mitscherlich}                & $t\to+\infty$         & $j+\beta(\rho-1)$ 		&$\beta	$		&$\displaystyle \frac{\beta (\rho+1)(j+\beta(\rho-1))}{j}$			&$<+\infty$		&$<+\infty$ 		
		& $<+\infty$\\
		& $\alpha\to0$                                                    & $j$				&$0$				&$0$				&$0$           &$0$          &$\frac{j-1}{j}$                            \\
		& $\alpha\to+\infty$ & $j+\beta(\rho-1)$ 		 &$\beta$     &$\displaystyle \frac{\beta (\rho+1)(j+\beta(\rho-1))}{j}$	&$<+\infty$&$<+\infty$&$<+\infty$	                         \\
		& $\beta\to 0$                                                     & $j$ 	&$0$			&$0$			&$0$			&$0$	&$\frac{j-1}{j}$                           \\                   
	\end{tabular}
		
		\label{tab:Tabella11}
	\end{table}

	\begin{figure}[t]
		\centering
		\subfigure[]{\includegraphics[scale=0.35]{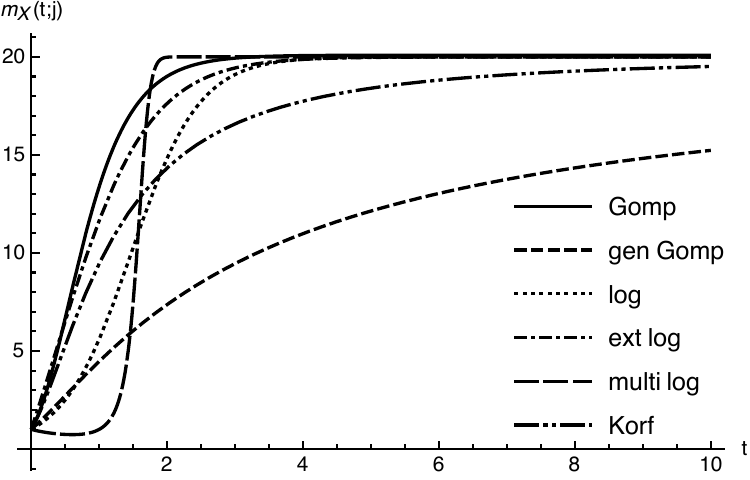}}\qquad
		\subfigure[]{\includegraphics[scale=0.35]{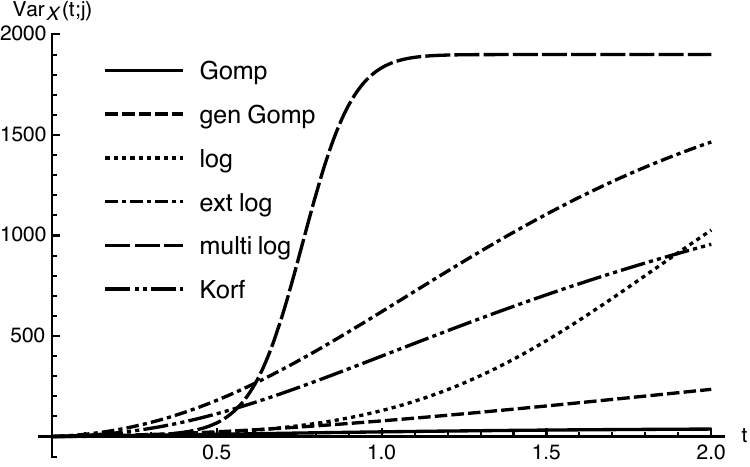}}\qquad
		\subfigure[]{\includegraphics[scale=0.35]{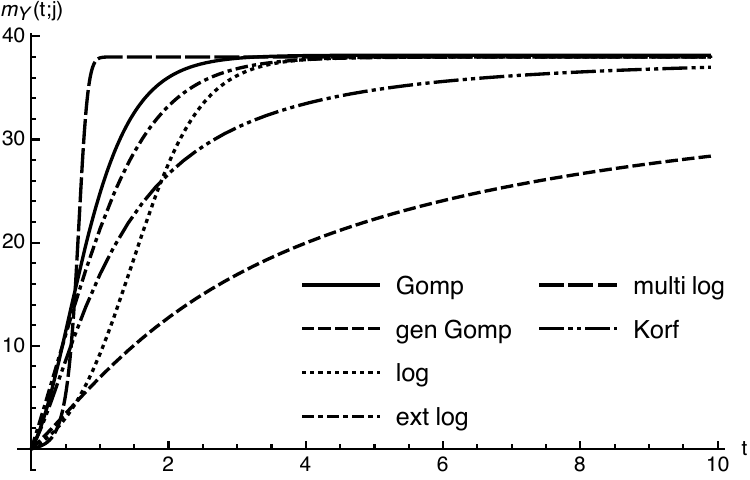}}\\
		\subfigure[]{\includegraphics[scale=0.35]{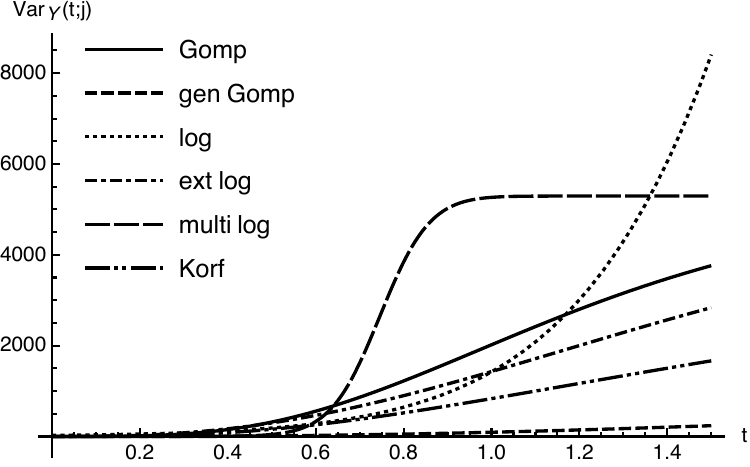}}\qquad
		\subfigure[]{\includegraphics[scale=0.35]{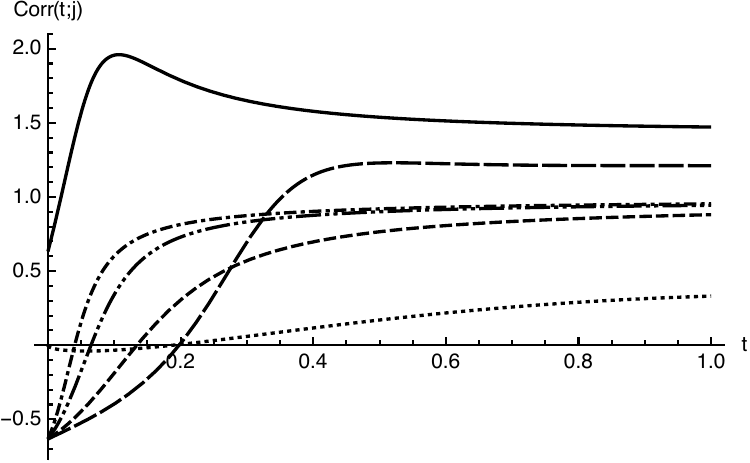}}\qquad
		\subfigure[]{\includegraphics[scale=0.35]{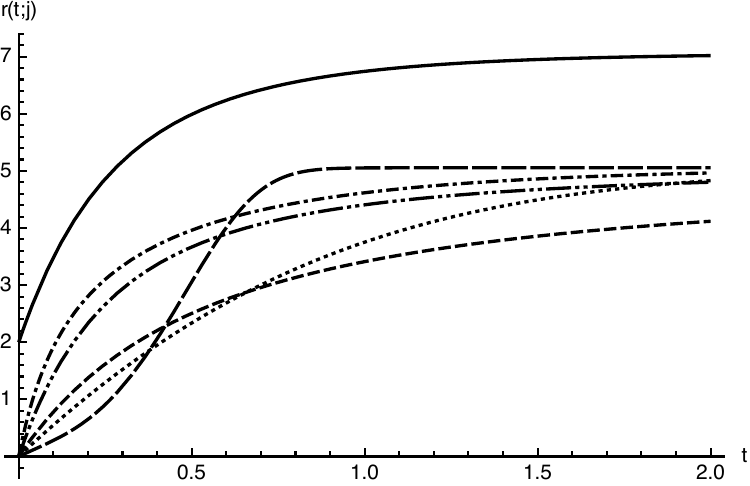}}
		\caption{The conditional means $m_X(t;j)$ and $m_Y(t;j)$, the conditional variances $Var_X(t;j)$ and $Var_Y(t;j)$, the conditional correlation coefficient $Corr(t;j)$ and the conditional coefficient $r(t;j)$ with $\rho=1.5$, $j=1$, $\alpha=3,\beta=2$ for Gompertz case, $A=\log 20, b=1$ for generalized Gompertz case,  $r=2,C=20$ for logistic case, $N=20, \varepsilon=0.5$ for the extended logistic case, $Q_\beta(t)=t-t^2+10t^3-t^4, C=20$ for the multisigmoidal logistic case and $\alpha=2\log20, \beta=2$ for the Korf case.}
		\label{fig:Figure23}
	\end{figure}

	\begin{sidewaystable}[]
		\caption{The limit behaviour of the conditional means $m_X(t;j)$ and $m_Y(t;j)$, the conditional variance $Var_X(t;j)$ and $Var_Y(t;j)$, the conditional covariance $Cov(t;j)$ and the conditional   correlation index $r(t;j)$.}
		\tiny
		\begin{tabular}{l|l|llllll}
			growth curve                             & \multicolumn{1}{c|}{conditions}                                  & \multicolumn{1}{c}{$m_X(t;j)$} & \multicolumn{1}{c}{$m_Y(t;j)$} & \multicolumn{1}{c}{$Var_X(t;j)$} & \multicolumn{1}{c}{$Var_Y(t;j)$} & \multicolumn{1}{c}{$Cov(t;j)$} & \multicolumn{1}{c}{$r(t;j)$} \\ \hline
			\multirow{5}{*}{(i) Gompertz}                & $t\to+\infty$         & $je^\alpha$ 		&$\displaystyle \frac{j(e^\alpha-1)}{\rho-1}$			&$\displaystyle \frac{je^\alpha(\rho+1)(e^\alpha-1)}{\rho-1}$			&$\mathcal V_{1,1}(j)$		&$\mathcal C_{1,1}(j)$ 		&$\mathcal R_{1,1}(j)$\\
			& $\alpha\to0$                                                    & $j$				&$0$				&$0$				&$\mathcal V_{1,2}(t)$           &$0$          &$1+\frac 1j$                            \\
			& $\alpha\to+\infty$ & $+\infty$ &$+\infty$ &$+\infty$ 			&$0$ 					&$+\infty$      &$1+\frac{3}{j}+\frac{2}{j(\rho-1)}$                              \\
			& $\beta\to 0$                                                     & $j$ 	&$0$			&$0$			&$0$			&$0$	&$1+\frac 1j$                           \\
			& $\beta\to+\infty$                                             &  $je^\alpha$ 			&$\displaystyle \frac{j(e^\alpha-1)}{\rho-1}$		&$\displaystyle \frac{je^\alpha(\rho+1)(e^\alpha-1)}{\rho-1}$	            &$+\infty$			&$\mathcal C_{1,1}(j)$			&$\mathcal R_{1,1}(j)$                           \\ \hline
			\multirow{5}{*}{(ii) generalized Gompertz}    & $t\to+\infty$                                                    & $je^{A b}$                               &    $\frac{j}{\rho-1}\left(e^{Ab}-1\right)  $                       & $\frac{j(\rho+1)e^{Ab}(e^{Ab}-1)}{\rho-1}  $                               
			& $\mathcal V_{2,1}(b;j)$                               &$\mathcal C_{2,1}(j)$ &$\mathcal R_{2,1}(b;j)$                           \\
			& $A\to+\infty$                                                    &     $+\infty$                           &   $+\infty$                             &      $+\infty$                            &$+\infty$                                  &  $+\infty$                              &$\frac{1+j(\rho-1)+\rho}{j(\rho-1)}$                              \\
			& $A\to 0$                                                         & $j$                               & $0$                               &  $0$                                &   $0$                               &  $0$                          &  $\frac{j-1}{j}$                            \\
			& $b\to +\infty$                                                   & $je^{At}$                               &$\frac{j(e^{At}-1)}{\rho-1} $                               &   $\frac{j(\rho+1)e^{At}(e^{At}-1)}{\rho-1}  $                                     &    $\mathcal V_{2,1}(t;j)$                                  & $\mathcal C_{2,2}(t;j)$                                  &$\mathcal R_{2,1}(t;j)$                       \\
			& $b\to 0$                                                         & $j$                               & $0$                               &$0$                                  &    $0$                              &  $0$                              &   $\frac{j-1}{j}$                           \\ \hline
			\multirow{4}{*}{(iii) logistic}                & $t\to+\infty$                                                                           & $C$ 					&$\frac{C-j}{\rho-1}$								&$\displaystyle \frac{C(\rho+1)(C/j-1)}{\rho-1}$			&$+\infty$								&$\mathcal{C}_{3,1}(j)$ 		&$\mathcal R_{3,1}(j)$ \\
			& $C\to+\infty$                                      & $je^{rt}$ 				&$\frac{j(e^{rt}-1)}{\rho-1}$ 				&$\displaystyle \frac{je^{rt}(e^{rt}-1)(1+\rho)}{\rho-1}$ 			&$\mathcal V_{3,1}(t;j)$ 					&$\mathcal C_{3,2}(j)$      &$\mathcal R_{3,2}(t;j)$                              \\
			& $r\to 0$                                       & $j$ 					&$0$														&$0$																						&$\frac{j(1+\rho+j(\rho-1))}{(\rho-1)^3}$			&$0$	&$1-\frac 1j$                              \\
			& $r\to + \infty$                               &$C$ 					&$\displaystyle \frac{C-j}{\rho-1}$		&$\displaystyle \frac{C(\rho+1)(C/j-1)}{\rho-1}$		            &$+\infty$			&$\mathcal C_{3,1}(j)$			&$\mathcal R_{3,1}(j)$                              \\ \hline
			\multirow{4}{*}{(iv) extended logistic}       & $t\to+\infty$                                                    & $N$ 																		&$\displaystyle \frac{N-j}{\rho-1}$														&$\displaystyle \frac{N(N-j)(1+\rho)}{j(\rho-1)}$			&$<+\infty$	 &        NA                        &   NA                           \\
			& $\varepsilon\to +1$                                              & $\displaystyle\frac{jNe^{2t}}{(e^{2t}-1)j+N}$		&$\displaystyle \frac{(e^{2t}-1)j(N-j)}{(\rho-1)(N+j(e^{2t}-1))}$		&$\displaystyle\frac{e^{2t}(e^{2t}-1)jN(N-j)(1+\rho)}{(\rho-1)((e^{2t}-1)j+N)^2}$				&$\mathcal{V}_{4,1}(t;j)$       &          NA                      &     NA                         \\
			& $\varepsilon\to -1$                                              & $\displaystyle\frac{N(j-2jt+2Nt)}{N-2jt+2Nt}$ 	&$\displaystyle \frac{2(N-j)^2t}{(\rho-1)(N-2jt+2Nt)}$ 						&$\displaystyle \frac{2(N-j)^2N(1+\rho)t(2Nt-j(2t-1))}{j(\rho-1)(N-2jt+2Nt)^2}$ 	&$<+\infty$ &  NA                              &           NA                   \\
			& $N\to +\infty$                                                   &    $+\infty$                            &  $+\infty$                              &   $+\infty$                               &  $+\infty$                                & NA                               &        NA                      \\ \hline
			\multirow{4}{*}{(v) multisigmoidal logistic} & $t\to+\infty$                                                    &     $C$                           &$0$                                &  $0$                                & $<+\infty$                                 &  $<+\infty$                              &  $<+\infty$                            \\
			& $\beta_i\to-\infty$                                              &  $C$                              &          $0$                        &    $0$                              &  $<+\infty$                                &   $\frac{Cj}{\rho-1}$                             &                     $<+\infty$         \\
			& $\beta_i\to+\infty$                                              &  $0$                              &   $\frac{C-j}{\rho-1}$                                &     $\frac{C(C-j)(\rho+1)}{j(\rho-1)}$                           &   $<+\infty$                           &    $<+\infty$                          &      $<+\infty$                        \\
			& $C\to+\infty$                                                    &  $je^{Q_\beta(t)}$                              &  $\frac{j(e^{Q_\beta(t)-1})}{\rho-1} $                              &  $\displaystyle \frac{je^{-Q_\beta(t)}(e^{-Q_\beta(t)}-1)(1+\rho)}{\rho-1}$                                 &   $<+\infty$                               &  $<+\infty$                              &   $\frac{\frac{1+j(\rho-1)+\rho}{j}-\frac{2\rho Q_\beta(t)}{e^{Q_\beta(t)}-1}}{\rho-1}$                           \\ \hline
			\multirow{5}{*}{(vi) modified Korf}           & $t\to+\infty$                                                  & $je^{\alpha/\beta}$ 					&$\frac{j(e^{\alpha/\beta}-1)}{\rho-1}$								&$\frac{e^{\alpha/\beta}(e^{\alpha/\beta}-1)j(\rho+1)}{\rho-1}$			&$\mathcal V_{6,2}(j)$								&$\mathcal{C}_6(j)$ 		&$\mathcal R_{6,1}(j)$                              \\
			& $\alpha\to 0$                                                    & $j$						&$0$														&$0$																					&$0$           					&$0$          &$\frac{j-1}{j}$                            \\
			& $\alpha\to +\infty$                                             & $+\infty$ 				&$+\infty$ 				&$+\infty$ 			&$+\infty$ 					&$+\infty$      &$\frac{1+j(\rho-1)+\rho}{j(\rho-1)}$                        \\
			& $\beta\to 0$                                                    & $j(1+t)^{\alpha}$ 					&$\frac{j(-1+(1+t)^\alpha)}{\rho-1}$														&$\mathcal V_{1,1}(t;j)$																						&$\mathcal V_{6,3}(t;j)$			&$\mathcal C_7(t;j)$	&$\mathcal R_{6,2}(t;j)$                            \\
			& $\beta\to +\infty$                                         &  $j$ 					&$0$		&$0$		            &$+\infty$			&$0$			&$\frac{j-1}{j}$                             
		\end{tabular}
		\label{tab:Tabella8}
	\end{sidewaystable}

	\begin{table}[t]
	\centering
	\tiny
	\caption{Limit functions of the conditional variances, of the correlation coefficient and of the coefficient $r(t;j)$. Note that such auxiliary functions are needed only for certain cases.} \label{tab:Tabella10}
	\begin{tabular}{l|l}
		growth curve                      & variance limits \\ \hline
		\multirow{2}{*}{(i) Gompertz}      &   $\displaystyle \mathcal V_{1,1}(j):=\frac{\beta\left(-\alpha(1+\alpha)+(1+\alpha)e^{2j}+e^j(-1+\alpha^2-4\alpha j)\right)}{(\alpha-1)^3}  $           \\
		&  $ \displaystyle \mathcal V_{1,2}(t;j):=\beta\exp\left\{j-2je^{-\rho t}\right\}\left(\exp\left\{je^{-\rho t}\right\}-e^j\right)$              \\ \hline
		(ii) generalilzed Gompertz          & $\displaystyle \mathcal V_{2,1}(x;j)=\frac{j\left(e^{2Ax}(\rho+1)-\rho(\rho+1)+e^{Ax}\left(-1-4Ax\rho+\rho^2\right)\right)}{(\rho-1)^3} $               \\ \hline
		(iii) logistic                       &   $ \displaystyle \mathcal {V}_{3,1}(t;j):=j\frac{-\rho(1+\rho)+e^{rt}(\rho^2-4r\rho t -1)+e^{2rt}(2-j+\rho(2+j))}{(\rho-1)^3}  $           \\ \hline
		\multirow{2}{*}{(iv) extended logistic}            &     $\displaystyle \mathcal{V}_{4,1}(t;j):=\frac{j}{(\rho-1)^3(N+j(e^{2t}-1))^2}\left[(e^{2t}-1)(N-j)(1+\rho)((N-j)\rho\right.$\\
		&$\left.+e^{2t}(N+j\rho))\right. \left.+4e^{2t}N\left((e^{2t}-1)j+N\right)\rho\left(-2t+\log\frac{(e^{2t}-1)j+N}{N}\right)\right]$            \\ \hline
		\multirow{3}{*}{(vi) modified Korf} &  $ \displaystyle \mathcal V_{6,1}(t;j):=\frac{j(1+\rho)(1+t)^\alpha((1+t)^\alpha-1)}{\rho-1}    $          \\
		&        $\displaystyle\mathcal V_{6,2}(j):=\frac{j\left[-4\alpha e^{\alpha/\beta}\rho+\beta(e^{\alpha/\beta}-1)(1+\rho)(e^{\alpha/\beta}+\rho)\right]}{\beta(\rho-1)^3}$         \\
		&       $\displaystyle \mathcal V_{6,3}(t;j):=\frac{j\left[(1+\rho)((1+t)^\alpha-1)(\rho+(1+t)^\alpha)-4\alpha\rho(1+t)^\alpha\log(1+t)\right]}{(\rho-1)^3}   $      
	\end{tabular}
	\begin{tabular}{l|l}
		growth curve                              & correlation limits \\ \hline
		(i) Gompertz                               &   $\displaystyle \mathcal C_{1,1}(j):=je^\alpha\frac{1-\rho(3+2\alpha)+e^\alpha(3\rho-1)}{(\rho-1)^2}$                \\ \hline
		\multirow{2}{*}{(ii) generalilzed Gompertz} &     $\displaystyle \mathcal C_{2,1}(j):= \frac{je^{Ab}\left(-1+2j-(1+2Ab+2j)\rho+2^{Ab}(1+2j(\rho-1)+\rho)\right)}{(\rho-1)^2} $               \\
		&  $\displaystyle C_{2,2}(t;j):=\frac{je^{At}\left(-1+e^{At}(\rho+1)-\rho(1+2At)\right)}{(\rho-1)^2}$                  \\ \hline
		\multirow{2}{*}{(iii) logistic}              &  ${C}_{3,1}(j):=\frac{C\left((\rho+1)(C-j)-2\rho j \log C/j\right)}{j(\rho-1)^2}$                  \\
		&         $\displaystyle\mathcal{C}_{3,2}(t;j):=je^{rt}\frac{e^{rt}(\rho+1)-1-\rho(1+2rt)}{(\rho-1)^2}$           \\ \hline
		\multirow{2}{*}{(vi) modified Korf}         &     $C_{6,1}(j):=\frac{e^{\alpha/\beta}j(-2\alpha\rho+\beta(e^{\alpha/\beta}-1)(\rho+1))}{\beta(\rho-1)^2}$               \\
		&         $\displaystyle\mathcal C_{6,2}(t;j):=\frac{j(1+t)^\alpha\left((1+\rho)(-1+(1+t)^\alpha)-2\alpha\rho\log(1+t)\right)}{(\rho-1)^2}$          
	\end{tabular}
	\begin{tabular}{l|l}
		growth curve                              & $r(t;j)$ limits \\ \hline
		(i) Gompertz                               &     $\displaystyle \mathcal R_{1,1}(j):=\frac{1+j-(3+2\alpha+j)\rho+e^\alpha(-1+j(\rho-1)+3\rho)}{(e^\alpha-1)(\rho-1)}$               \\ \hline
		(ii) generalilzed Gompertz &    $ \displaystyle \mathcal R_{2,1}(x;j):=\frac{-2Ax\rho+(e^{Ax}-1)(1+j(\rho-1)+\rho)}{j(\rho-1)(e^{Ax}-1)}    $    \\ \hline
		\multirow{2}{*}{(iii) logistic}              &     $\displaystyle\mathcal{R}_{3,1}(j):=1+\frac{1+\rho}{j(\rho-1)-\frac{2\rho\log C/j}{(\rho-1)(C-j)}}$               \\
		&        $\mathcal{R}_{3,2}(t;j):=\frac{1-j+\rho\left(1-\frac{2rt}{e^{rt}-1}+j\right)}{j(\rho-1)}$            \\ \hline
		\multirow{2}{*}{(vi) modified Korf}         &      $\displaystyle\mathcal R_{6,1}(j):=\frac{1+j(\rho-1)+\rho+\frac{2\alpha\rho}{\beta(1-e^{\alpha/\beta})}}{j(\rho-1)}$              \\
		&      $ \mathcal R_{6,2}(t;j):=\frac{1+j(\rho-1)+\rho-\frac{2\alpha\rho\log(1+t)}{(1+t)^\alpha-1}}{j(\rho-1)}$   
	\end{tabular}

\end{table}

\begin{figure}[t]
	\centering
	\subfigure[]{\includegraphics[scale=0.35]{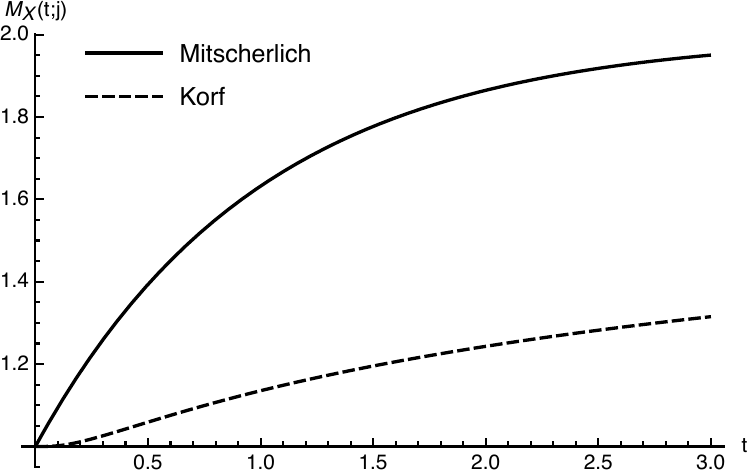}}\qquad
	\subfigure[]{\includegraphics[scale=0.35]{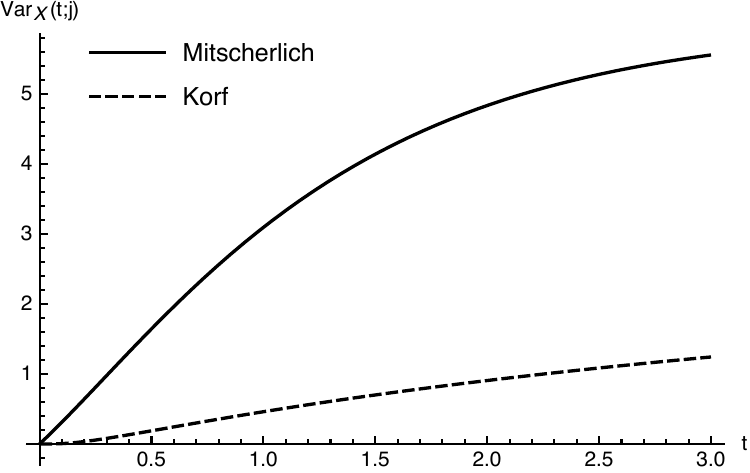}}\qquad
	\subfigure[]{\includegraphics[scale=0.35]{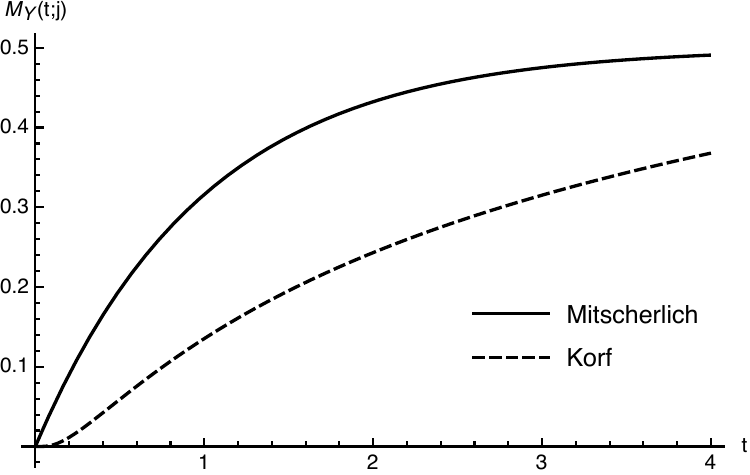}}\\
	\subfigure[]{\includegraphics[scale=0.35]{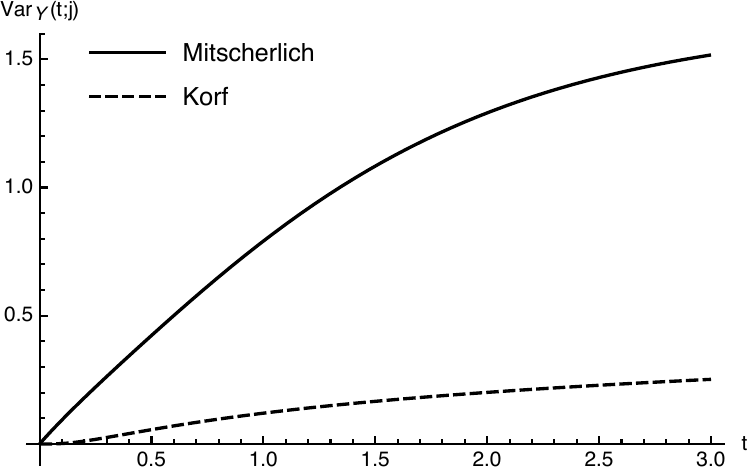}}\qquad
	\subfigure[]{\includegraphics[scale=0.35]{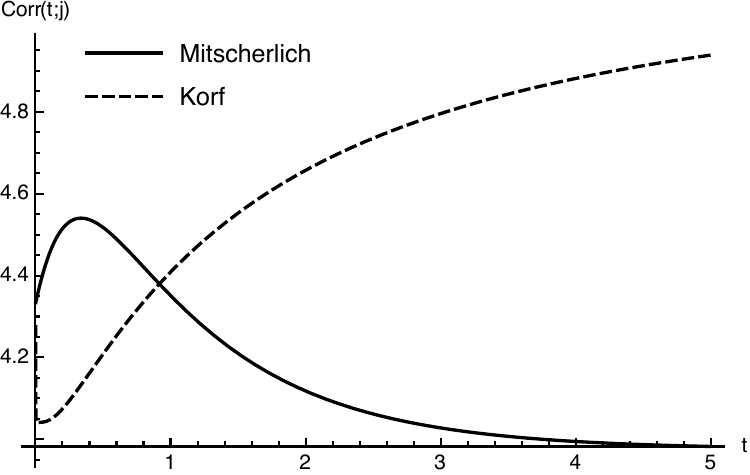}}\qquad
	\subfigure[]{\includegraphics[scale=0.35]{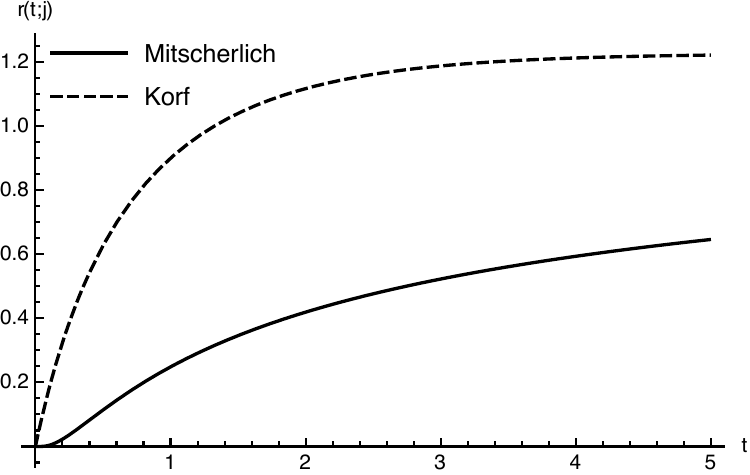}}
	\caption{The conditional means $m_X(t;j)$ and $m_Y(t;j)$, the conditional variances $Var_X(t;j)$ and $Var_Y(t;j)$, the conditional correlation coefficient $Corr(t;j)$ and the conditional coefficient $r(t;j)$ with $\rho=2$, $j=1$, $\alpha=1,\beta=0.5$.}
	\label{fig:Figure25}
\end{figure}

	\section{Application to real data}\label{Section8}
	In this section we consider an application to real data. The datasets come from \cite{linkdati} and concern the temporal diffusion of 5 different fake news on $\mathbb X$/Twitter. For each fake news, the number of tweets and retweets relating to it were counted with the relative time reference. Regarding the time scale, we use the time instant in which the first tweet was posted as the initial time, and number of days since the initial instant as the time unit. Once the temporal evolution of the fake news has been constructed, we use the growth curves analyzed in detail in Section \ref{Sec:special-cases} as candidate models. The parameters involved in the various definitions were estimated by imposing the minimization of the mean square error (MSE) or of the relative absolute error (RAE), as done in \cite{DiCrescenzoetal2022a}. Regarding the multisigmoidal logistic model, we chose $Q_\beta(t)$ as a fourth degree polynomial, since with repeated empirical tests this degree shows to be the right trade off between goodness of estimate and reduced number of parameters to estimate. We remember, in fact, that the presence of many parameters entails a high computational cost and the possibility of having overfit problems. Once all the parameters of interest are estimated, we select the best model as the one leading to the smallest MSE (or to the smallest RAE). The temporal evolution of fake news and the related estimating models are shown in Figure \ref{fig:Figure24}.

	\begin{figure}[t]
		\centering
		\subfigure[]{\includegraphics[scale=0.26]{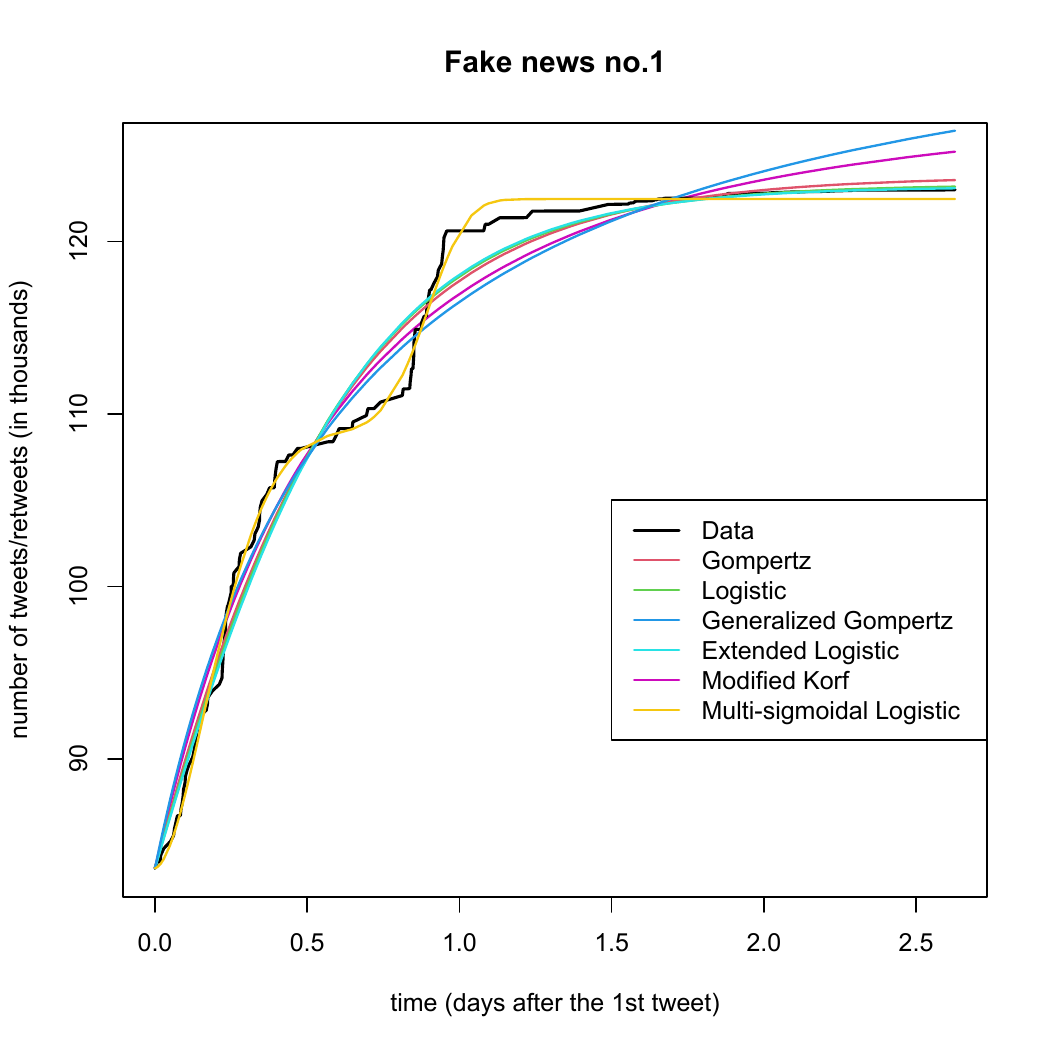}}\qquad
		\subfigure[]{\includegraphics[scale=0.26]{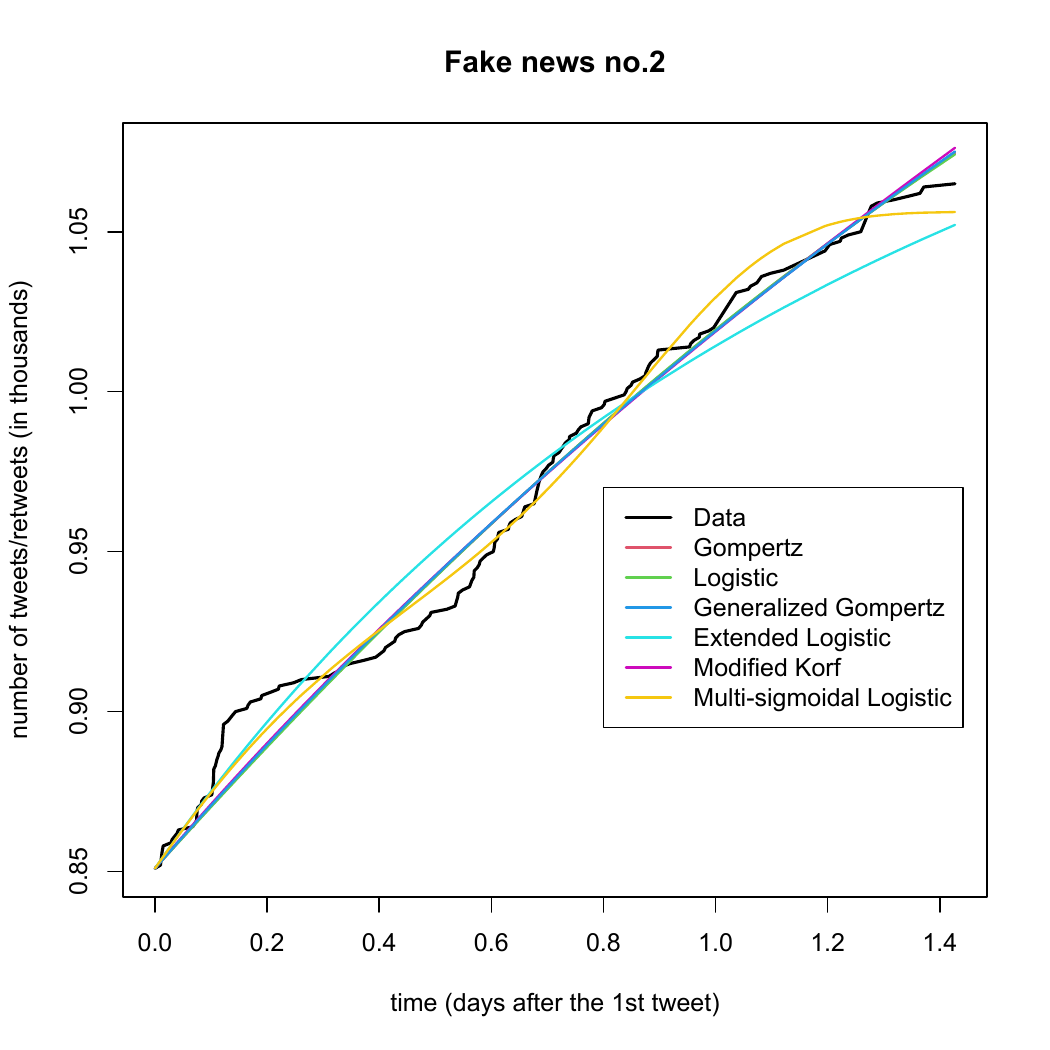}}\qquad
		\subfigure[]{\includegraphics[scale=0.26]{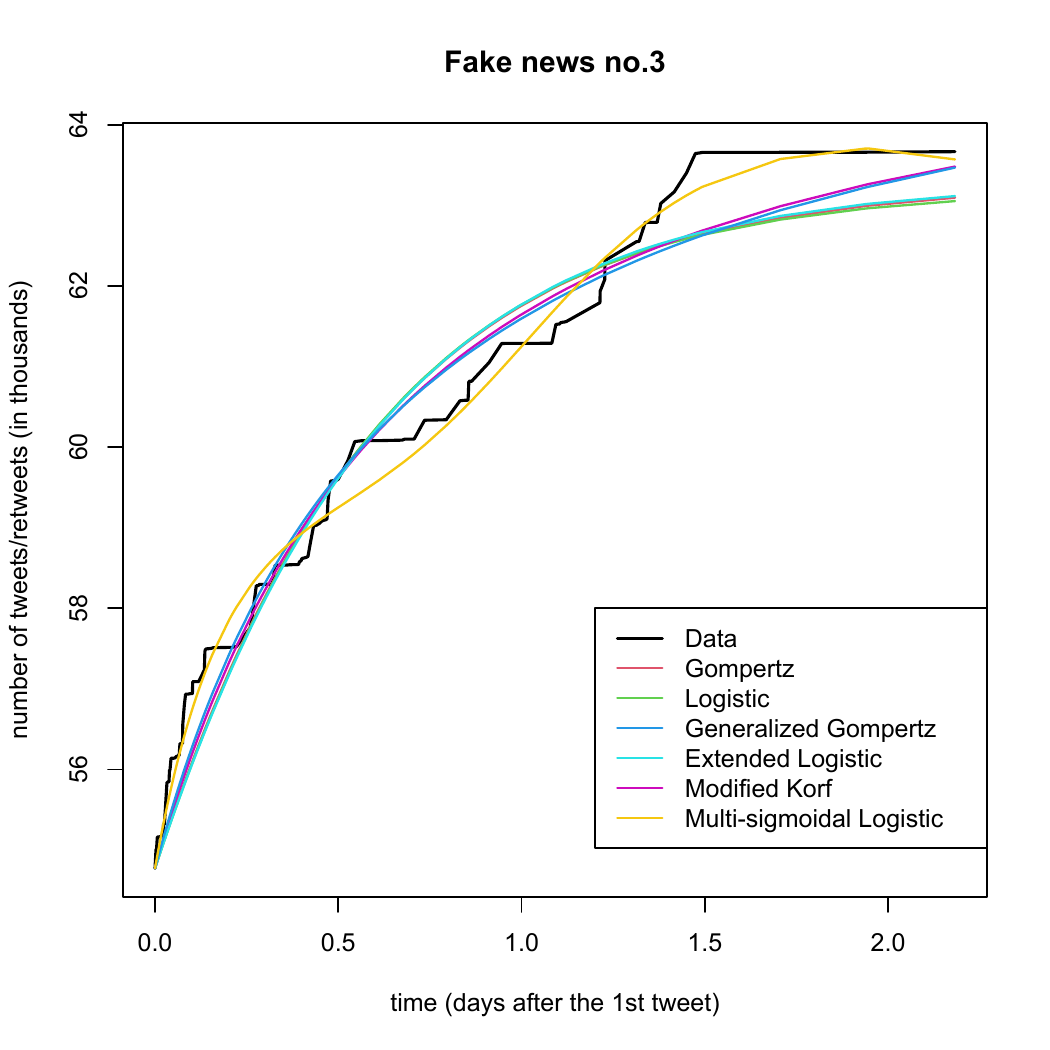}}\\
		\subfigure[]{\includegraphics[scale=0.26]{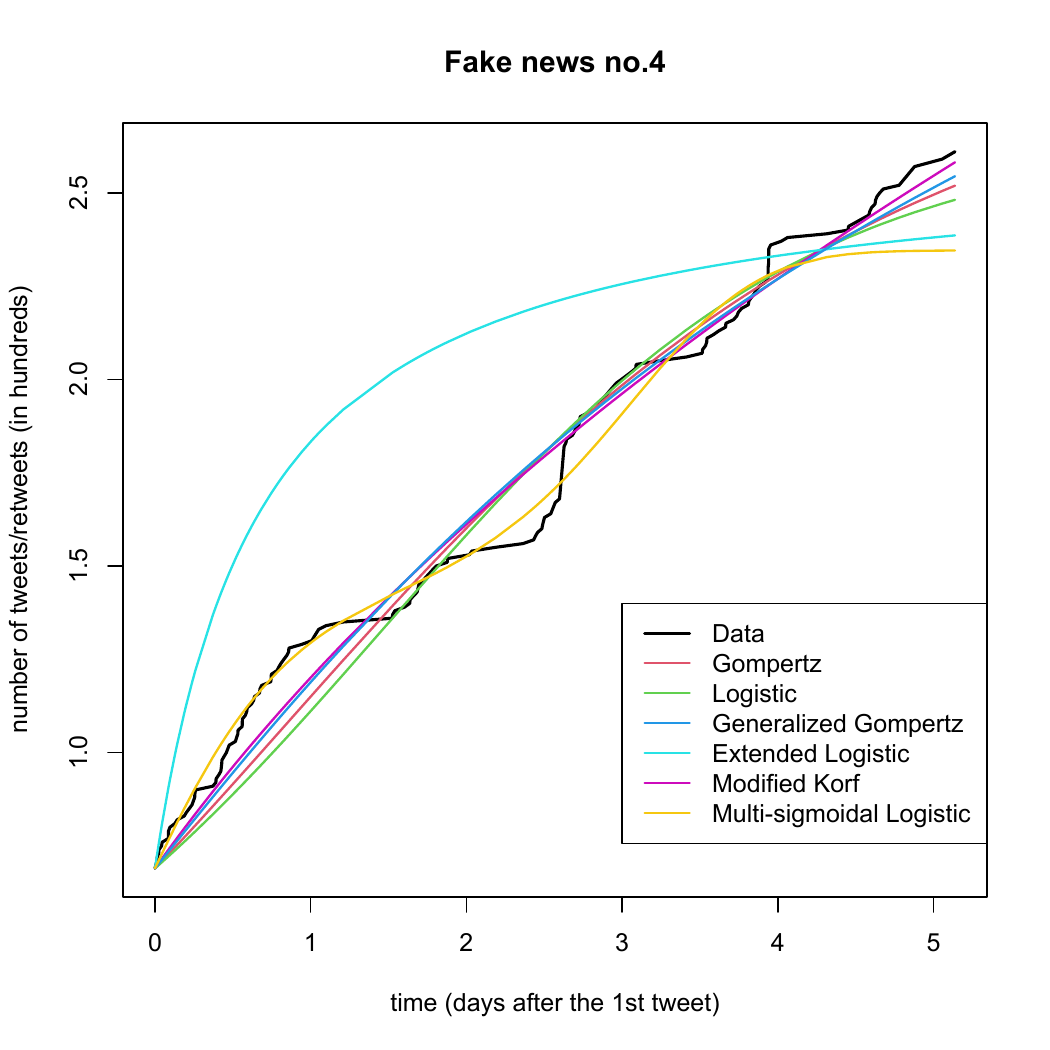}}\qquad
		\subfigure[]{\includegraphics[scale=0.26]{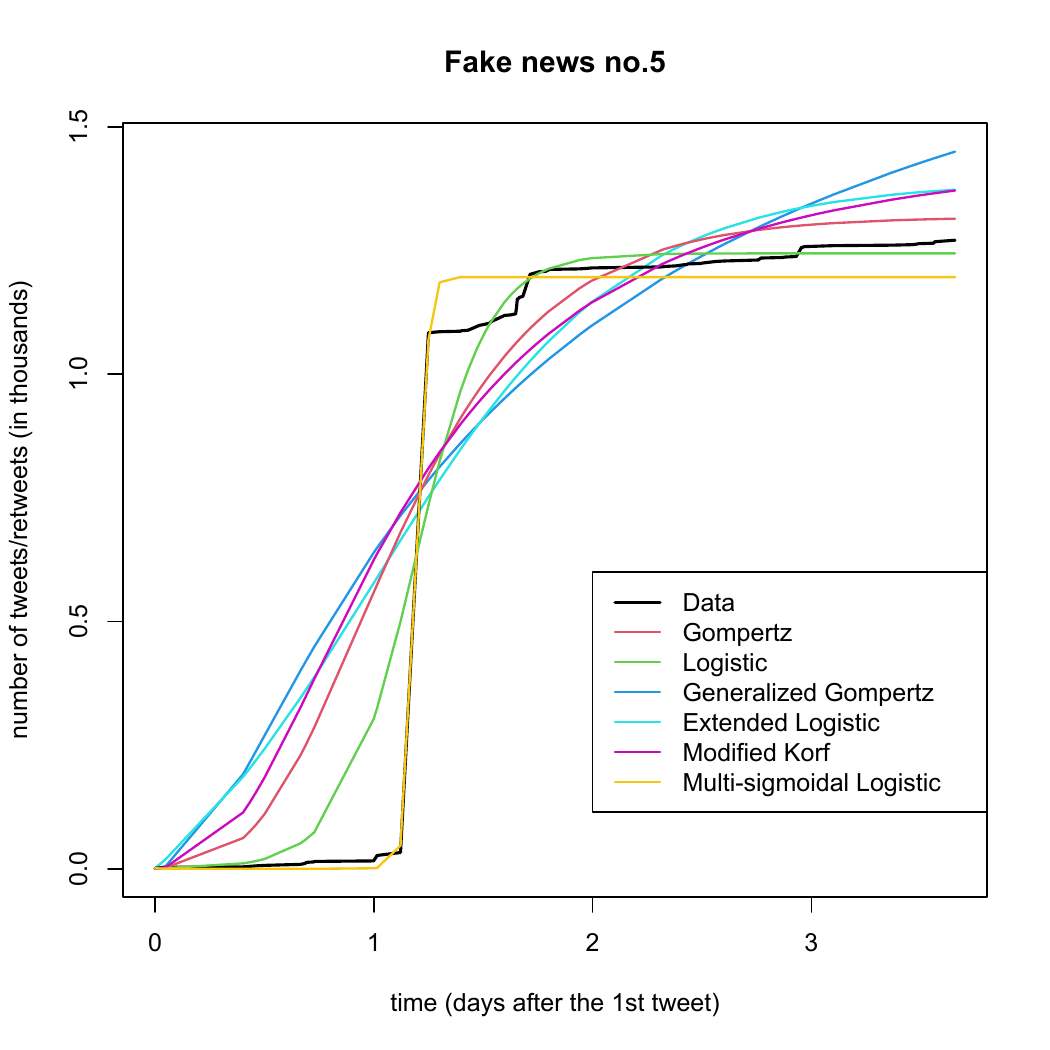}}
		\caption{The fake news time evolution and the corresponding fitting models $m_X(t;j)$.}
		\label{fig:Figure24}
	\end{figure}
	It is clear that, for all the considered datasets, the model that guarantees the best results is the multisigmoidal logistic one, as also confirmed by the quantitative results reported in Table \ref{tab:Tabella9}.
	However, using the multisigmoidal logistic model imply the estimation of a quite large number number of parameters (given by the degree of the polynomial + 2). Hence, one can be interested in fitting the data with a more simple model, i.e.\ one of the others studied in Section \ref{Sec:special-cases}. Considering this type of data approximation, the model which returns the best fit (i.e. the one with the smallest MSE or the smallest RAE) varies depending on the analyzed fake news. More specifically, for the dataset no.\ 1 the best model is the Gompertz one, for the dataset no.\ 2 it is the modified Korf, for the datasets no.\ 3-4 it is the generalized Gompertz, for the dataset no.\ 5 it is the logistic.
	\begin{table}[h]
		\caption{The mean square error (MSE) and the relative absolute error (RAE) between the data and the fitted models for dataset (DS) no.\ 1$\div$5. The minimum MSE and the minimum RAE for each dataset is underlined.}
		\label{tab:Tabella9}
		\tiny
		\centering
		\begin{tabular}{ll|llllll}
			& & case (i)						& case (ii)				& case(iii)						& case (iv) 					& case (v) 				& case (vi)\\
			& & Gompertz               & gen. Gompertz   & logistic               & ext. logistic      & multis. logistic &  mod. Korf          \\ \hline
			DS no.\ 1 &MSE & $2.71523$              & $3.53550$              & $2.86095$              & $2.99282$              & \underline{$0.48732$}               & $2.96641$              \\
			&RAE & $0.99484$              & $0.01460$              & $0.01192$              & $0.01190$              & \underline{$0.00467$}               & $0.01327$              \\ \hline
			DS no.\ 2 &MSE & $6.94419\cdot10^{-5}$ 	& $6.91156\cdot10^{-5}$ 		& $6.96783\cdot10^{-5}$ 	& $1.22240\cdot10^{-4}$	 & \underline{$4.49001\cdot10^{-5}$}  		& $6.83370\cdot10^{-5}$ \\
			&RAE & $0.00695$              & $0.00698$              & $0.00693$              & $0.00935$              & \underline{$0.00526$}               & $0.00703$              \\ \hline
			DS no.\ 3 &MSE & $0.30543$              & $0.21796$              & $0.31618$              & $0.31420$              & \underline{$0.06927$}               & $0.24032$              \\
			&RAE & $0.96458$              & $0.00663$              & $0.00798$              & $0.00797$              & \underline{$0.00422$}               & $0.00697$              \\ \hline
			DS no.\ 4 &MSE & $0.00958$              & $0.00733$              & $0.01274$              & $0.15590$              & \underline{$0.00501$}               & $0.00594$              \\
			&RAE & $0.05806$              & $0.04751$              & $0.06657$              & $0.76616$              & \underline{$0.03131$}               & $0.04303$              \\ \hline
			DS no.\ 5 &MSE & $0.02391$              & $0.04795$              & $0.00838$              & $0.03883$              & \underline{$0.00280$}               & $0.03449$    \\
			&RAE & $0.55334$              & $0.57466$              & $0.34989$              & $0.57569$              & \underline{$0.09370$}               & $0.58865$              \\   \hline       
		\end{tabular}
	\end{table}
	Finally, for any fake news we also calculate the corresponding mean $m_Y(t;j)$ considering different values for $\rho$. The plots are shown in Figure \ref{fig:Figure13nov23}. Note that in certain cases, it was not possible to proceed with the determination of the mean $m_Y(t;j)$ using the multi-sigmoidal logistic approximation which, we remember, is the best for any fake news. The reason is owing to numerical causes, in fact the multi-sigmoidal logistic curve involves the calculations of very small quantities which are mistakenly treated by the computer as $-\infty$. Therefore, in these cases, it is preferred to use as the approximating curve the one that determines the smallest MSE or the smallest RAE after the multi-sigmoidal logistics. It can be interesting, as a future development, studying a procedure to approximate also the coefficient of proportionality $\rho$.
	\par 
 As mentioned before, in the analysis conducted in this section both the MSE and the RAE decrease with increasing model complexity, expressed in terms of the number of involved parameters. Precisely, the high number of parameters produces models that tend to fit the data better. However, this characteristic can lead to models prone to overfitting, i.e. models that do not generalize well to new data. Therefore, measures of goodness that take into account the complexity of the model in addition to its ability to fit the data are often used in applied contexts. This is the case of the Akaike Information Criterion and the Bayesian Information Criterion  goodness-of-fit measures that introduce a penalty based on the complexity of the model. This line of inquiry based on such information criteria has been successfully pursued in \cite{DiCrescenzoetal2022a}.
 To use the aforementioned criteria, however, the expression of the likelihood function is needed. 
 Unfortunately, it is not available in the present case. Calculating the likelihood function for the considered model is quite hard, 
 and thus it is beyond the specific scope of this work. However, one can consider proposing a diffusive approximation of the current model, as performed 
 in other studies on similar topics. This can lead to a more comprehensive statistical analysis in a future research work. 
	\begin{figure}[t]
		\centering
		\subfigure[]{\includegraphics[scale=0.26]{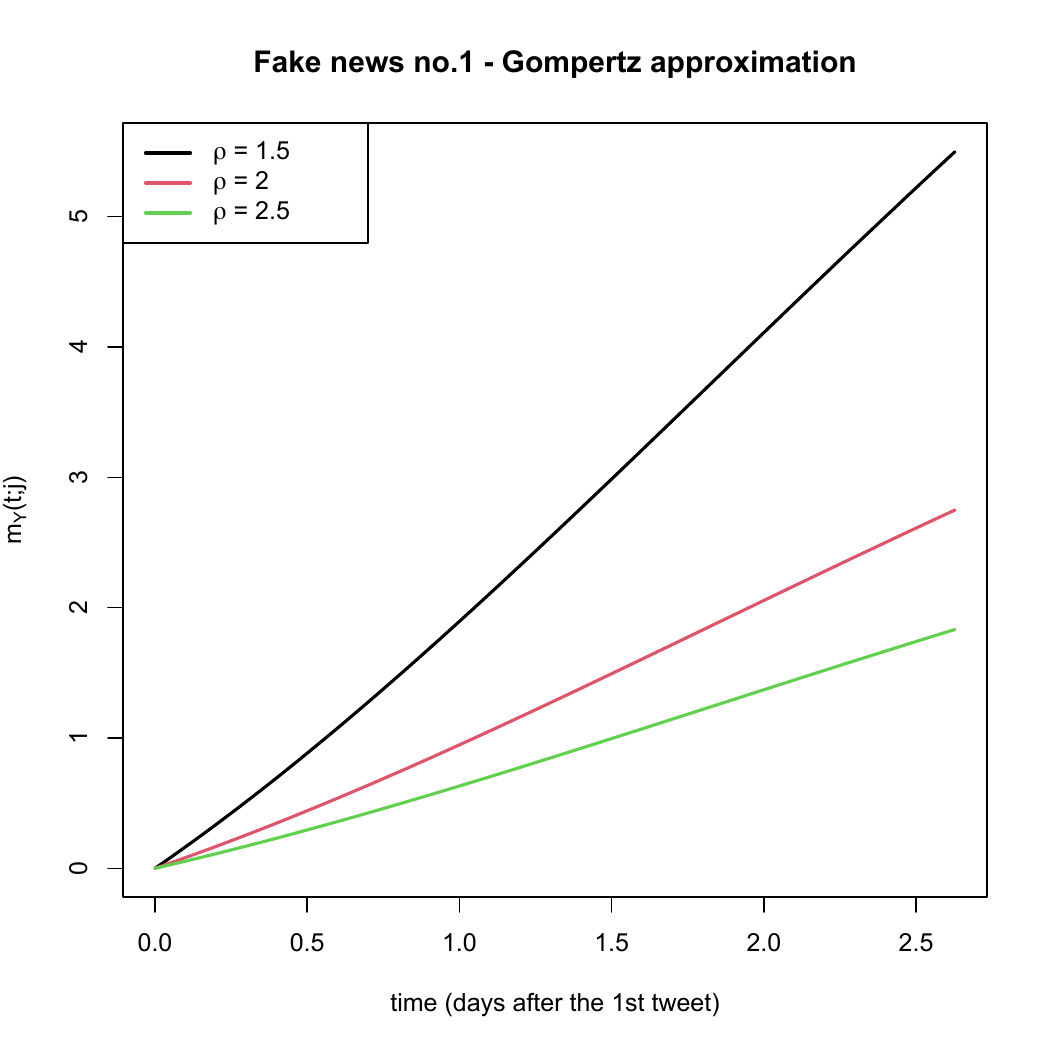}}\qquad
		\subfigure[]{\includegraphics[scale=0.26]{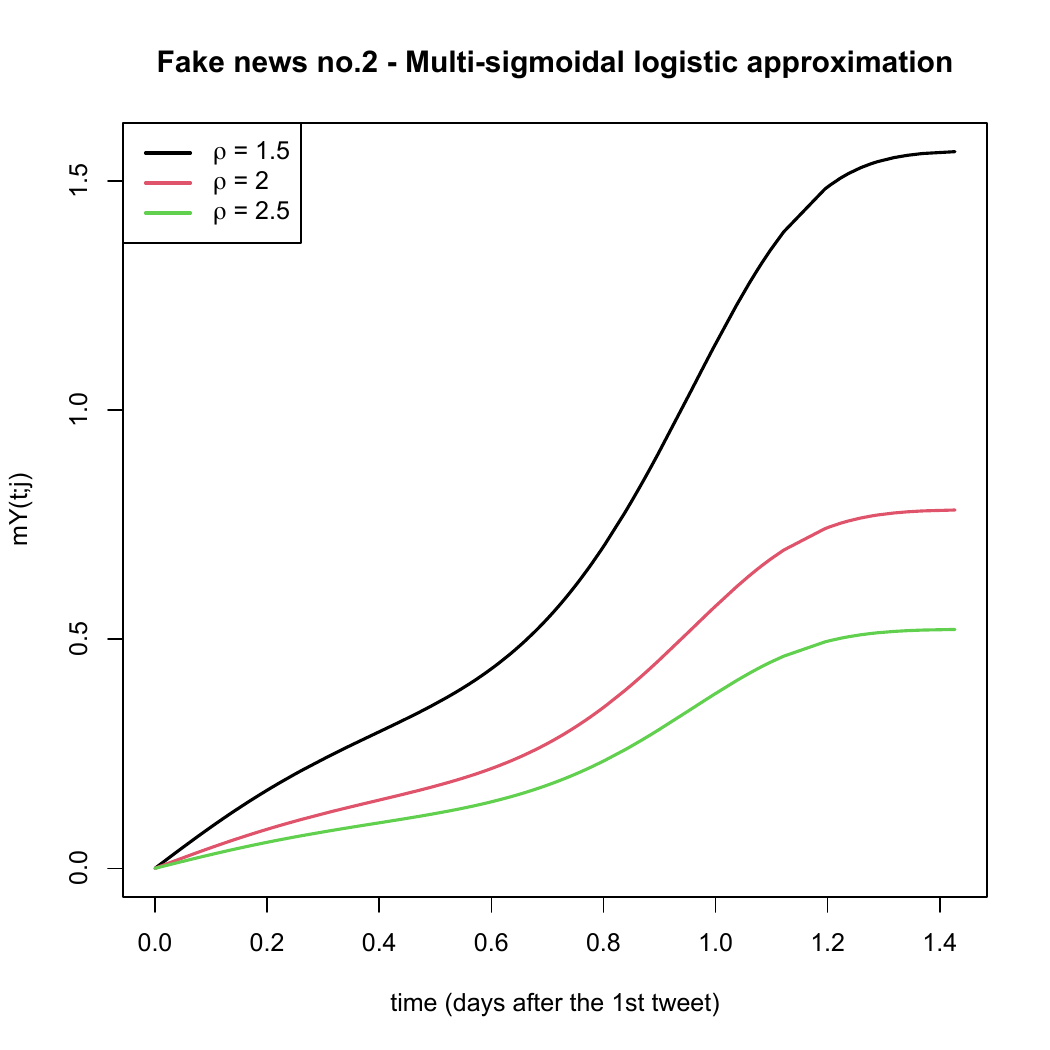}}\qquad
		\subfigure[]{\includegraphics[scale=0.26]{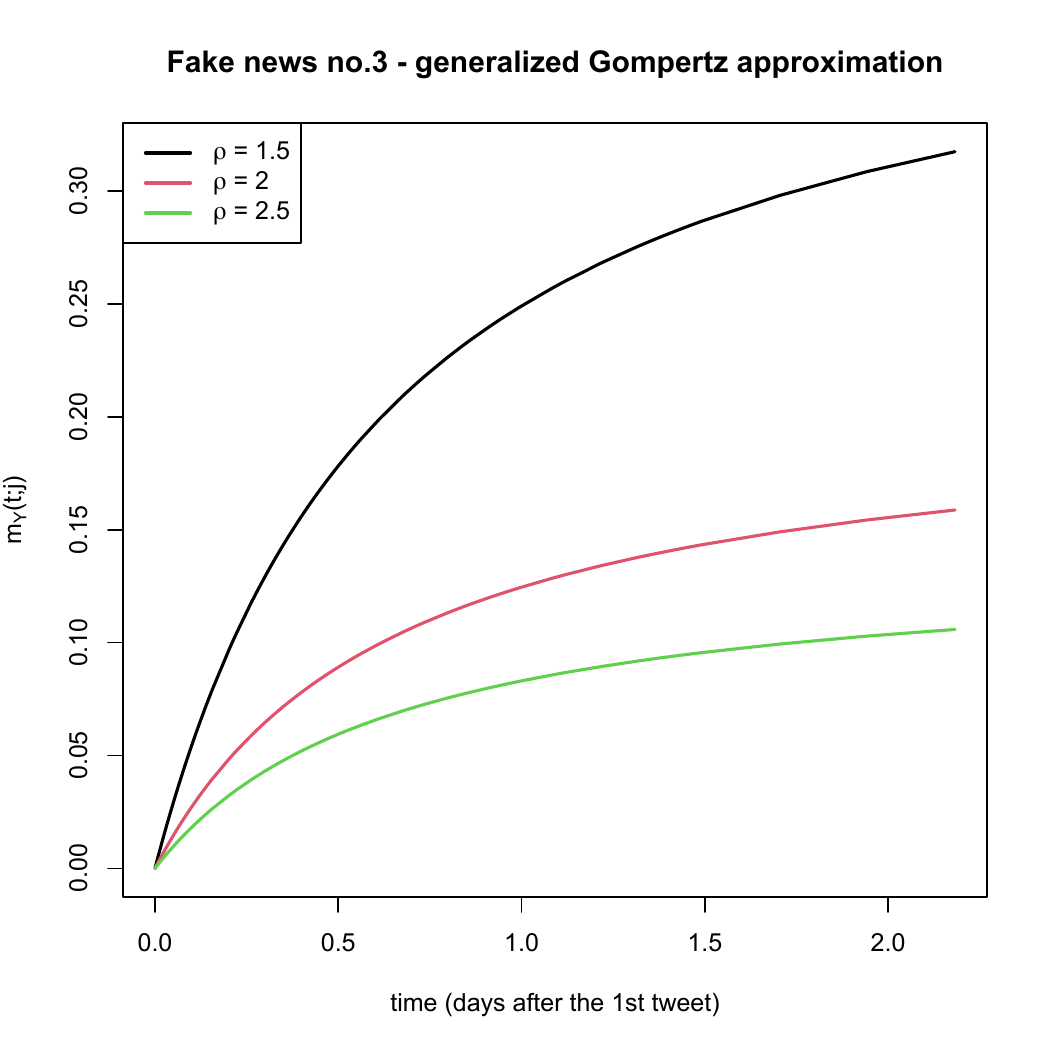}}\\
		\subfigure[]{\includegraphics[scale=0.26]{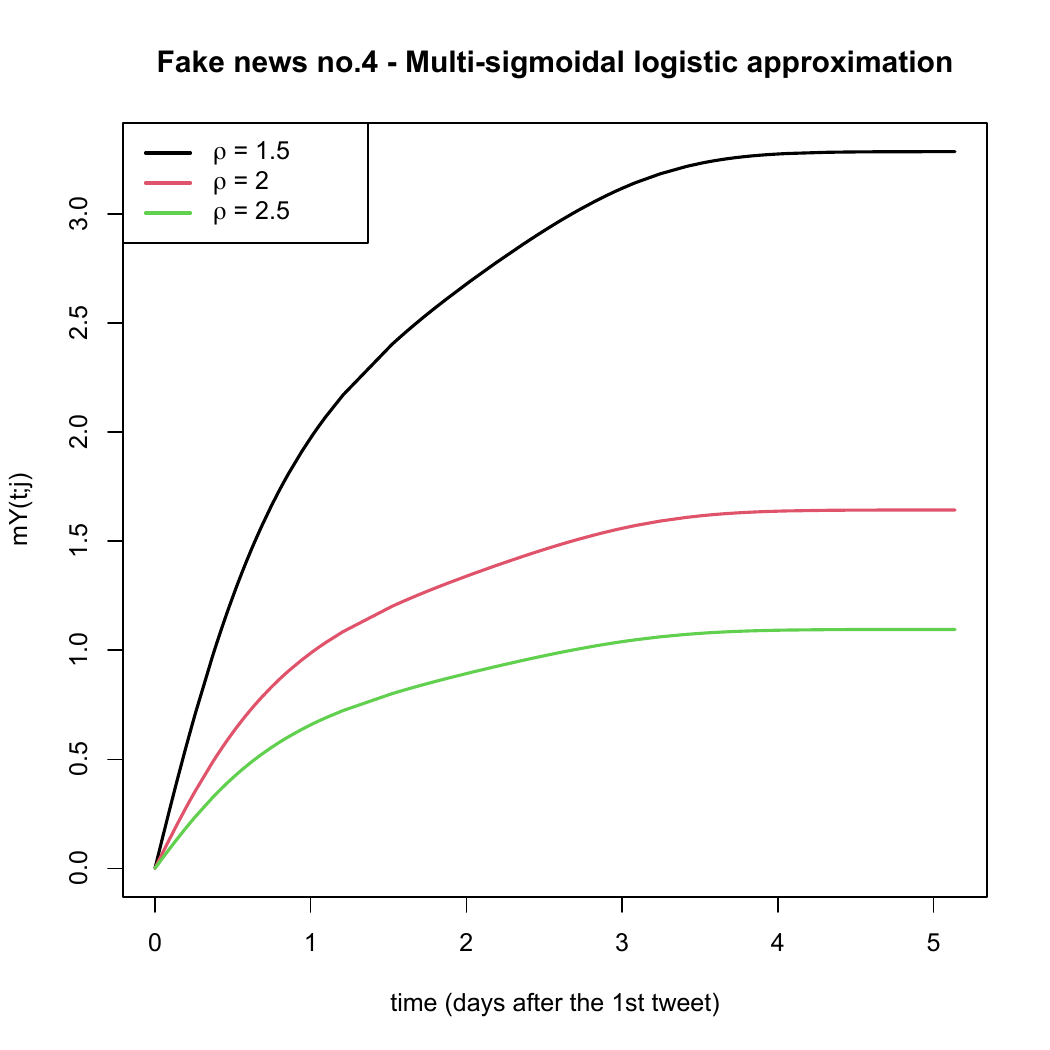}}\qquad
		\subfigure[]{\includegraphics[scale=0.26]{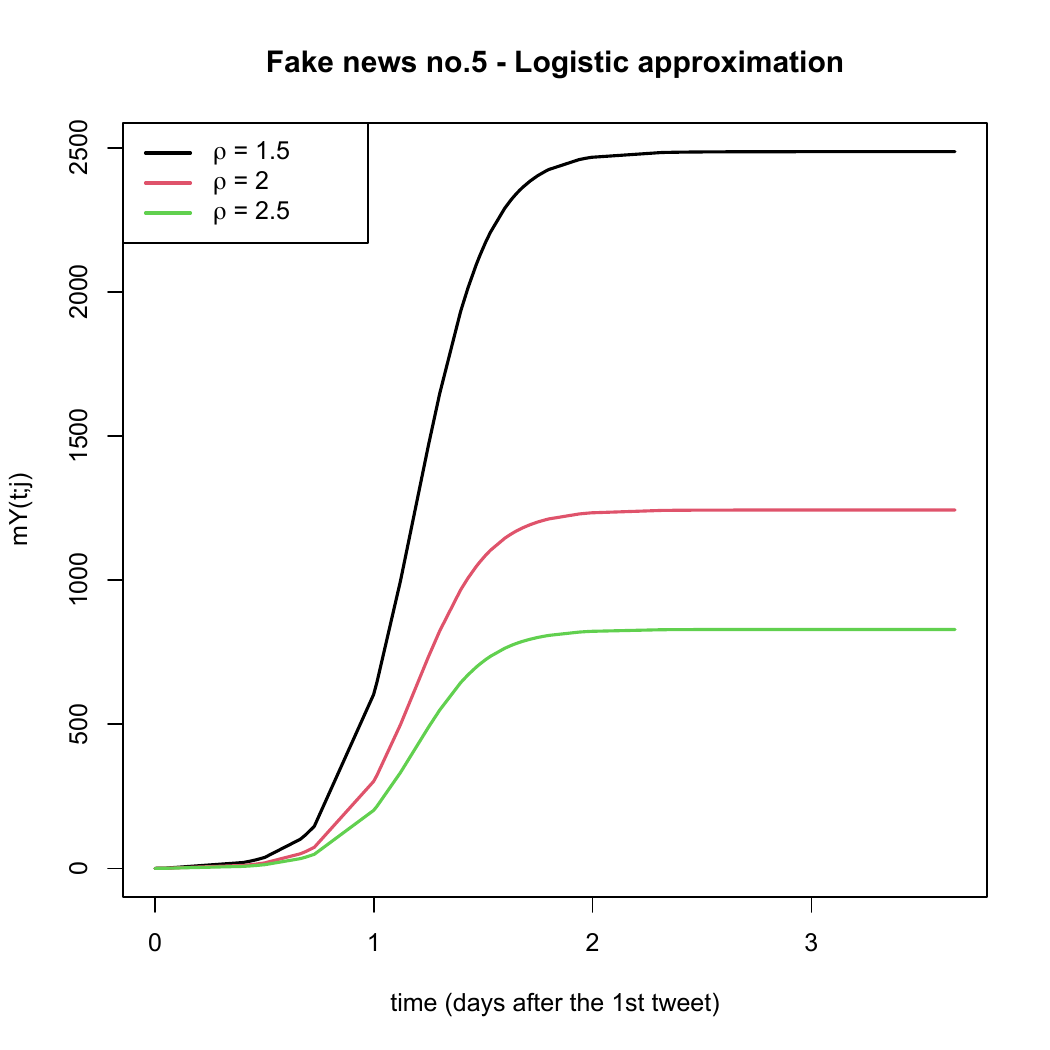}}
		\caption{The approximating $m_Y(t;j)$ for different values of $\rho$.}
		\label{fig:Figure13nov23}
	\end{figure}
	\section{Conclusions}
	 The paper is devoted to modeling the temporal propagation of fake news in online social networks. In particular, we introduced a probabilistic model represented by a two-dimensional birth-death process whose first component shows the number of individuals who knows the fake news and intends to spread it, while the second component shows the number of individuals who has forgotten the fake news or has lost the intention to spread it. The transitions which regulate the dynamics of the two-dimensional process are determined by temporal non-homogeneous individual birth and death rates and they are proportional to the size of the first component. {The process, introduced in this paper for the first time, was then analyzed through the study of the p.g.f.} The partial differential equation whose solution is represented by the p.g.f. is not easily solvable in the general case but has been solved in the case of temporal homogeneous rates and in the case of temporal non-homogeneous proportional rates. Starting from the p.g.f., the first and second order moments of both the first and second components, the mixed moments and the covariance of the two-dimensional process were determined. Furthermore, to study the correlation between the two components, a new index was studied, represented by the ratio between the mean of the product of the two components and the product of the means. {The definition of such adimensional index is one of the novelties introduced by the paper.} Considering that the first component of the process represents the number of individuals who knows the news and intends to spread it in a large population, we considered the conditions that make the mean of this first component equal to sigmoidal growth curves well known in the literature. Finally, we proposed an application to real data considering the data linked to the temporal diffusion of 5 fake news. In particular, by comparing the mean square error 
and the relative absolute error we determined the model that best fits the data of the case study. Such application makes the introduced two-dimensional process interesting from a modeling point of view. 
	In certain real situations relating to the spread of fake news, the information of interest can be obtained efficiently only via numerical methods  (see, for instance, D'Ambrosio \textit{et al.}\ \cite{Dambrosioetal2021}). {With the idea of providing closed-form results, which is one of the main  strength of the paper, the proposed stochastic process is a significantly simplified model.}
	\par
	In order to make the model more realistic for the analyzed datasets, it may be useful to consider cases where the mean of $X(t)$ coincides with other growth curves, for example the T-growth model (cf.\ Barrera \textit{et al.} \cite{Barreraetal2021}).
	Moreover, analyzing the time that the process takes to reach a critical size is interesting in certain application contexts; therefore, an analysis aimed at the first passage time (as done in Albano \textit{et al.} \cite{Albanoetal2023}) represents one of the possible future developments. 
	The model can be reformulated by giving a different interpretation to the components $X(t)$ and $Y(t)$. Thus, the stochastic process may be used to analyze the trend of other data sets, coming from different real contexts, such as the epidemiological one. 
	Furthermore, it will be of interest to divide the population into more disjoint compartments and analyze the resulting model both from a deterministic and a stochastic point of view. 
	\par
	We conclude by remarking that in order to stop the spread of the fake news it is necessary to make the number of inactives large and the number of spreaders tending to 0. In the present model, such behavior emerges when the intensity of the spreaders $\lambda$ is smaller than the intensity of the inactives $\mu$, i.e.\ when $\rho<1$ in the case of proportional intensities. 
	\par
 Moreover, we point out that the model characterized by intensities \eqref{trans-rates} can also be useful for describing the spread of news in general (not necessarily just fake news) within an infinite population. In this case, one might expect a non-sigmoidal growth in average, since the news could potentially spread throughout the entire population which is assumed to be large. Such generalization is out of the scope of the present paper but can be the object of a future research.  
		\section*{Acknowledgements}
	The authors are members of the research group GNCS of INdAM (Istituto Nazionale di Alta Matematica). 
	This work is partially  funded by the `European Union -- Next Generation EU' through 
	MUR-PRIN 2022, project 2022XZSAFN ``Anomalous Phenomena on Regular and Irregular Domains: Approximating Complexity for the Applied Sciences'', and MUR-PRIN 2022 PNRR, project P2022XSF5H ``Stochastic Models in Biomathematics and Applications''.
	
			\section*{Data availability statement}
	The data that support the findings of Section \ref{Section8} of this study are available in Zenodo at https://zenodo.org/records/5225338, reference number \cite{linkdati}.  

	\section*{Conflict of interest}
		The authors declare no potential conflict of interests.
		\section*{ORCID}
			A. Di Crescenzo: http//orcid.org/0000-0003-4751-7341\\
			P. Paraggio: https://orcid.org/0000-0002-3308-7937

\end{document}